\documentclass{article}
\usepackage{amsmath}
\usepackage[shortlabels]{enumitem}
\usepackage{amssymb}
\usepackage{mathrsfs}
\usepackage{extarrows}
\usepackage{MnSymbol}
\usepackage{tikz}
\usepackage{tikz-cd}
\usetikzlibrary{shapes,positioning,patterns,calc}
\usepackage{amsthm}
\usepackage[backend=bibtex]{biblatex}
\usepackage[boxsize=1.2em, centerboxes]{ytableau}

\def\AA{{\mathbb A}}

\def\FF{{\mathbb F}} 
\def\GG{{\mathbb G}}

\def\PP{{\mathbb P}}

\def\WW{{\mathbb W}}

\def\ZZ{{\mathbb Z}}

\newcommand{\ul}{\underline}

\def\Bscr{{\mathscr B}}

\def\Hscr{{\mathscr H}}

\def\Nscr{{\mathscr N}}

\def\Sscr{{\mathscr S}}

\def\Xscr{{\mathscr X}}
\def\Yscr{{\mathscr Y}}
\def\Zscr{{\mathscr Z}}

\def\cal{\mathcal}
\def\cA{{\cal A}}

\def\cC{{\cal C}}
\def\cD{{\cal D}}

\def\cF{{\cal F}}
\def\cG{{\cal G}}

\def\cJ{{\cal J}}

\def\cO{{\cal O}}
\def\cP{{\cal P}}

\def\Spec{\operatorname{Spec}}

\def\Ann{\operatorname{Ann}}
\def\Ext{\operatorname{Ext}}
\def\End{\operatorname{End}}

\def\Hom{\operatorname{Hom}}

\def\Proj{\operatorname{Proj}}
\def\Lie{\operatorname{Lie}}

\def\stmod{\operatorname{stmod}}

\def\0ol{{\bar 0}}
\def\1ol{{\bar 1}}
\def\2ol{{\bar 2}}
\def\ol2{{\bar 2}}
\def\3ol{{\bar 3}}
\def\4ol{{\bar 4}}
\def\5ol{{\bar 5}}
\def\6ol{{\bar 6}}
\def\7ol{{\bar 7}}
\def\8ol{{\bar 8}}
\def\9ol{{\bar 9}}

\def\P2Skly{\PP^2_{Skly}}

\def\coker{\operatorname {coker}}

\def\End{\operatorname {End}}
\def\Ext{\operatorname {Ext}}

\def\ker{\operatorname {ker}}

\def\Lie{\operatorname {Lie}}

\def\supp{\operatorname {supp}}

\def\Ann{\operatorname{Ann}}
\def\Aut{\operatorname{Aut}}

\def\dim{\operatorname{dim}}

\def\End{\operatorname{End}}

\def\Ext{\operatorname{Ext}}

\def\Hom{\operatorname{Hom}}

\def\id{\operatorname{id}}

\def\max{\operatorname{max}}

\def\min{\operatorname{min}}
\def\mod{{\sf mod}\ }

\def\Proj{\operatorname{Proj}}

\def\rep{{\sf rep\ }}

\def\Spec{\operatorname{Spec}}

\def\stmod{\operatorname{stmod}}

\def\uHom{\operatorname{\underline{Hom}}}

\def\ul1{\operatorname{\underline{1}}}

\def\stmod{\operatorname{stmod}\nolimits}

\def\fB{{\mathfrak B}}

\def\fE{{\mathfrak E}}
\def\fF{{\mathfrak F}}

\def\fL{{\mathfrak L}}
\def\fM{{\mathfrak M}}
\def\fN{{\mathfrak N}}
\def\fO{{\mathfrak O}}

\def\fg{{\mathfrak g}}

\def\fh{{\mathfrak h}}

\def\fn{{\mathfrak n}}
\def\fp{{\mathfrak p}}
\def\fq{{\mathfrak q}}

\def\ft{{\mathfrak t}}

\def\cVec{{\sf Vec}\ }

\def\dirlim{\mathop{\vtop{\baselineskip -100pt\lineskip -1pt\lineskiplimit 0pt
\setbox0\hbox{lim}\copy0\hbox to \wd0{\rightarrowfill}}}\limits}
\def\invlim{\mathop{\vtop{\baselineskip -100pt\lineskip -1pt\lineskiplimit 0pt
\setbox0\hbox{lim}\copy0\hbox to \wd0{\leftarrowfill}}}\limits}

\def\I11{{1 \kern -0.8pt \! \mbox{l}}}
\def\mumu{{\mu\kern-4.2pt\mu}}
\def\bfmu{{\mu\kern-4.2pt\mu}}
\def\2slash{\backslash \! \backslash}

\def\boxtimes{\setbox0\hbox{$\Box$}\copy0\kern-\wd0\hbox{$\times$}}

\def\Supph\operatorname{Supph}

\def\sl{\mathfrak{sl}}

\theoremstyle{definition}

\numberwithin{equation}{subsection}

\newtheorem{thm}[equation]{Theorem}
\newtheorem*{thm*}{Theorem}
\newtheorem{prop}[equation]{Proposition}
\newtheorem{definition}[equation]{Definition}
\newtheorem{example}[equation]{Example}
\newtheorem{expo}[equation]{}

\newtheorem{lemma}[equation]{Lemma}
\newtheorem*{lemma*}{Lemma}
\newtheorem{cor}[equation]{Corollary}

\newtheorem{conjecture}[equation]{Conjecture}
\newtheorem*{conjecture*}{Conjecture}

\title{A tensor-triangular property for categories of representations of restricted Lie algebras}

\author{Justin Bloom}
\date{December 2024}

\begin{document}

\maketitle

\begin{abstract}
    We define a property for restricted Lie algebras in terms of cohomological support and tensor-triangular geometry of their categories of representations.
    By Tannakian reconstruction, the different symmetric tensor category structures on the underlying linear category of representations of a restricted Lie algebra correspond to different cocommutative Hopf algebra structures on the restricted enveloping algebra. In turn this equates together the linear categories of representations for various group scheme structures.
    The tensor triangular spectrum, for representations of a restricted Lie algebra, is known to be isomorphic to the scheme of 1-parameter subgroups of the infinitesimal group scheme structure associated to the Lie algebra. Points in the spectrum of a tensor triangulated category correspond to minimal radical thick tensor-ideals, provided the spectrum is noetherian, as is known in our case of finite group schemes. When the group scheme structure changes from the Lie algebra structure, a set of subgroups can still yield points of the spectrum, but there may not be enough to cover the spectrum.
    Restricted Lie algebras satisfy our property if, for each group scheme structure, the remaining set of subgroups correspond to minimal radical thick tensor-ideals having identical Green-ring structure to that of the original Lie algebra.
    Some small examples of algebras of finite and tame representation type satisfying our property are given. We show that no abelian restricted Lie algebra of wild representation type may have our property. We conjecture that satisfying our property is equivalent to having finite or tame representation type. 
\end{abstract}

\tableofcontents

\section{Introduction}
\subsection{Overview}

In this paper we define Property PC for Jacobson's restricted Lie algebras \cite{Jac41} in terms of cohomological support for their category of restricted representations. The purpose of finding Lie algebras with this property is to illustrate how tensor-triangular geometry can be used to find constraints imposed on Green ring calculations for families of symmetric tensor categories.

We adopt the convention throughout that all representations of a restricted Lie algebra are restricted, and that $k$ is always an algebraically closed field of characteristic $p > 0$. We also adopt the convention that all finite group schemes over $k$ are affine. For a finite group scheme $G$ over $k$, we denote by $kG$ the \emph{group algebra,} also called the measure algebra, a cocommutative Hopf algebra which is dual to the coordinate algebra $k[G] = \cO(G)$ of the affine scheme $G.$ We denote $A^*$ the $k$-linear dual of a $k$-space $A$. We see any finite dimensional cocommutative Hopf algebra $A$ is $kG$ for some finite group scheme $G = \Spec A^*$.

We will review the notions of cohomological support for finite group schemes in Section \ref{background}. The work of Benson, Carlson, and Rickard \cite{BCR96} on finite groups, and later the works of Friedlander and Pevtsova \cite{FrPev07}, \cite{FrPev05}, and of Benson, Iyengar, Krause, and Pevtsova \cite{BIKP15} on finite group schemes, establishes an equivalence between the data of cohomological support and tensor-triangular support (Balmer \cite{Balmer05}) for the stable module category. Thus, for a finite group scheme $G$, a closed point $\fp \in \Proj H^*(G, k)$ gives a subcategory $\cC(\fp)$ of finite modules supported only at the singleton $\{\fp\}$, and $\cC(\fp)$ is a \emph{minimal thick $\otimes$-ideal.} What is interesting for our purposes is that when two finite group schemes both have a group algebra isomorphic to a given associative algebra $A$, there is an equivalence of $k$-linear categories of representations, but not of tensor categories. Yet the subcategories $\cC(\fp)$ (or any subcategory with support in a fixed subset of $\Proj H^*(A, k)$) remain $\otimes$-ideal independent of which $\otimes$ is chosen, so long as the one dimensional module $k$ comes from a fixed augmentation $A \to k$. This may be seen by noticing how the ring structure on $H^*(A, k) = \Ext^*_A(k, k)$ is not dependent on a choice of Hopf algebra structure, and similarly the module action on cohomology $H^*(A, M) = \Ext^*_A(k, M)$ for any $A$-module $M$. We investigate for which group schemes and which points $\fp$ the Green ring structure on $\cC(\fp)$ may be known to not change between group schemes. 

When $\fg$ is a restricted Lie algebra of dimension $r$, recall that restricted representations of $\fg$ coincide with modules over the restricted enveloping algebra $A = u(\fg)$, which has dimension $p^r.$ The algebra $A$ is canonically a cocommutative Hopf algebra by taking elements of $\fg$ to be primitive. Thus there is a canonical group scheme $\widetilde G$ with group algebra $A,$ which we call the \emph{infinitesimal group scheme} corresponding to $\fg$. Given any cocommutative Hopf algebra structure $\Delta : A \to A \otimes A,$ we have a corresponding group scheme $G$ with group algebra $A$, and a corresponding tensor product $\otimes$ for modules over $A$, which are now also representations of $G.$ We will always denote by $\widetilde\Delta, \widetilde \otimes$ the canonical Hopf algebra comultiplication and tensor product for the Lie algebra $\fg$, or equivalently the infinitesimal group scheme $\widetilde G.$ We will only consider Hopf algebra structures sharing a counit $A \to k,$ so that $k$ is a fixed $A$-module acting as the monoidal unit with respect to any $\otimes.$
\begin{definition}\label{mainProperty}
    Let $\fg$ be a restricted Lie algebra over $k$, and $A = u(\fg)$. For finite group schemes $G$ coming from a Hopf algebra structure on $A$, let $\otimes$ denote the tensor product of $G$-representations. 
    
    The algebra $\fg$ is said to satisfy \emph{Property PC} if, for any such finite group scheme $G$, and any $\fp \in \Proj H^*(A, k),$ the tensor products $\otimes, \widetilde\otimes$ induce identical Green ring structures on the ideal $\cC(\fp)$ iff the prime $\fp$ is \emph{noble} (\ref{nobleDef}) for $G$. 
\end{definition}

We include, as Proposition \ref{tautnoble}, a consequence of the work of Bendel, Friedlander, Parshall, and Suslin \cite{FrPar87}, \cite{FrPar86}, \cite{SFB97}, which tells us that all homogeneous primes of $\Proj H^*(\fg, k)$ arise from $1$-dimensional Lie subalgebras of $\fg$. The idea of homogeneous primes arising from subgroups of a group scheme is generalized to our notion of a \emph{noble} prime (\ref{nobleDef}). In this way,  Proposition \ref{tautnoble} may be restated to say that every prime is noble for the infinitesimal group scheme corresponding to the restricted Lie algebra. 

In context of Property PC, we see that deforming the comultiplication structure on the restricted enveloping algebra may change the group scheme in such a way that a given homogeneous prime of cohomology is no longer noble. It is at these \emph{ignoble} primes where we expect the Green ring structure on the subcategory $\cC(\fp)$ to always change from its original structure from the Lie algebra. This we call Property PB, which in conjunction with its converse Property PA, becomes Property PC (see \ref{PABC}). 

Theorem \ref{introKlein} below is proven in Section \ref{sectionKlein} and shows there exists an algebra of tame representation type satisfying Property PC. 
\begin{thm}\label{introKlein}
Let $\fg$ be the 2-dimensional abelian Lie algebra with trivial restriction, i.e $\fg^{[p]} = 0$, over $k$ with characteristic $p = 2$. Then $\fg$ satisfies Property PC.
\end{thm}
Theorem \ref{intro2dimthm} below is proven in Section \ref{sectionp2lemma} and provides an algebra of wild representation type in each odd characteristic, all failing to satisfy Property PC.  
\begin{thm}\label{intro2dimthm} Let $\fg$ be the 2-dimensional abelian Lie algebra over $k$ with trivial restriction, i.e. $\fg^{[p]} = 0$. 
    If $p > 2$ then $\fg$ does not satisfy Property PA.
    In other words, there exists a group scheme $G$ as in Definition \ref{mainProperty} and a point $\fp$ which is noble for $G$, and modules $V, W \in \cC(\fp)$ such that $V\otimes W$ is not isomorphic to $V\widetilde\otimes W.$ So $\fg$ does not satisfy Property PC.
\end{thm}

The restricted Lie algebra $\fg$ of \ref{introKlein} and \ref{intro2dimthm} corresponds to the infinitesimal group scheme $\GG_{a(1)}^2.$ The enveloping algebra $u(\fg)$ is isomorphic as an associative algebra also to the group algebra $k(\ZZ / p)^2$ and to the group algebra for two other distinct group schemes over $k$ up to isomorphism. The classification of Hopf algebra structures on $u(\fg) = k[x,y]/(x^p, y^p)$ was first given as a corollary of the classification of all connected Hopf algebras of dimension $p^2$, by X. Wang \cite{XWang13}.

We state Conjecture \ref{conj3}, that having wild representation type is equivalent to failing to satisfy Property PC, and provide more examples of restriced Lie algebras having wild representation type, while not satisfying Property PC, in Section \ref{sectionWild}. We have found no restricted Lie algebra failing to satisfy Property PB, leading us to Conjecture \ref{conj4}.

It is fundamental to our use of noble points that the following algebra, of finite representation type, also satisfies Property PC. This is discussed further in \ref{exGa1}.
\begin{example}\label{introGa1}
    Let $\fg$ be the one dimensional Lie algebra over $k$ with restriction $\fg^{[p]} = 0.$ Then $\fg$ satisfies Property PC. In fact, there is only one homogeneous prime in $\Proj H^*(\fg, k) = \{\fp\}$ and we find that all group schemes $G$ with $kG \cong u(\fg)$ have that $\fp$ is noble for $G$, and that they all define identical Green rings.
\end{example}

In order to affirm that a given restricted Lie algebra $\fg$ of tame or finite representation type satisfies Property PC, we resort to using an explicit classification of all cocommutative Hopf algebra structures on the universal enveloping algebra $u(\fg).$
The work of Nguyen, Ng, L. Wang, and X. Wang, \cite{NgW23}, \cite{NWW15}, \cite{NWW16}, \cite{WW14}, \cite{XWang13}, \cite{XWang15} has classified connected and pointed Hopf algebras of small dimension. As a corollary there are many small dimensional Lie algebras $\fg$ for which $u(\fg)$ is a local ring (such $\fg$ are called unipotent), with all Hopf algebra structures on $u(\fg)$ known to be dual to one of the connected Hopf algebras of the above mentioned authors. All local Hopf algebras of order $p^2$ in characteristic $p$ are given explicitly as a corollary in \cite{XWang13} and we use this directly in Section \ref{sectionKlein} to affirm Property PC for a tame algebra over a field of characteristic $p=2$. 

Our Property PC pertains to cocommutative Hopf algebra \emph{structures} on a given finite augmented algebra, which we may think of as points of an affine algebraic set $\Hscr,$ whereas the classifications of the above mentioned authors yield only a complete set of representatives of the orbit space $\Hscr / \sim$ under the action of twisting by augmented algebra automorphisms (\ref{considerthequotient}). Our technique for proving Theorem \ref{introKlein} consists of first showing the expected behavior for modules supported at noble and ignoble points for a complete set of representatives of $\Hscr / \sim$, as first classified by X. Wang in \cite{XWang13}. 
In doing so it is shown that both the tame algebra of Theorem \ref{introKlein} and the wild algebra of Theorem \ref{intro2dimthm} satisfy a weaker version of Property PB, quantified only over a complete set of representatives of $\Hscr / \sim.$ This is extended in the tame case $p = 2$ by using that if $M$ is any finite representation supported at a point $\fp$, we find $M^\varphi \cong M$ whenever $\varphi \in \Aut(u(\fg))$ is an algebra automorphism fixing the point $\fp \in \Proj H^*(\fg, k).$

Our paper concludes with Section \ref{sectionWild}, an expos{\'e} on augmented algebra automorphisms, and on induced modules, for the restricted enveloping algebra of some Lie algebras having wild representation type. The group of (augmented) automorphisms $\Aut(u(\fg))$ acts on $\Hscr$, on the isomorphism classes $\pi_0(\rep \fg),$ and on the projective scheme $\Proj H^*(\fg, k).$
Our technique, for proving that a Lie algebra $\fg$ fails to satisfy Property PC, is to provide nontrivial elements of the quotient \[\Aut(u(\fg))^\fp/ \Aut(u(\fg))^{\pi_0(\cC(\fp))}\] of isotropy subgroups $\Aut(u(\fg))^{\fp} , \Aut(u(\fg))^{\pi_0(\cC(\fp))}$ of $  \Aut(u(\fg)),$ for a choice of point $\fp \in \Proj H^*(\fg, k).$ Given that an automorphism fixes each isomorphism class in $\pi_0(\cC(\fp)),$ it must also fix the point $\fp$. It is conjectured that the converse only fails for algebras $\fg$ of wild representation type; this is in turn informed by the conjecture that the continuous parameter for indecomposables of any tame Lie algebra is always realizable as support.

For the right choice of isotropy $\varphi \in \Aut(u(\fg))^\fp$, we find that twisting the Lie Hopf algebra structure gives a tensor product $\widetilde\otimes^\varphi$ such that $V \widetilde\otimes^\varphi W \not\cong V \widetilde\otimes W,$ where $V, W$ are a choice of modules having support $\{\fp\},$ induced from a subalgebra of $\fg$. We show how to produce such an isotropy for any restricted Lie algebra with a wild abelian Lie algebra as a direct summand. We also produce such an isotropy in odd characteristic for the Heisenberg Lie algebra, which contains a wild elementary abelian Lie subalgebra, but not as a direct summand. 

\subsection{Why Lie algebras?}\label{ourpurpose}
For our study of certain families of tensor-categories, we will review some tools from the tensor-triangular geometry of the stable module category of a group scheme in section \ref{defsection}. 
For one, we establish that subcategories supported at a subset of homogeneous primes of cohomology are closed under tensor product no matter which tensor product is chosen! Further, two modules will have tensor product a projective module if and only if they have disjoint support, and again their support and hence this property is independent of which tensor product is chosen. So by our account, comparing Green rings structures on the same underlying abelian category leads directly to tensor-triangular geometry. Many computations of the product in a tame category follow from abstract impositions from the theory of support before a product is even chosen.

Now we direct the reader to Proposition \ref{tautnoble}. This proposition states how every prime is noble for $\widetilde G$, whenever $\widetilde G$ is the infinitesimal group scheme corresponding to a restricted Lie algebra $\fg$. This is a necessary condition for the given restricted Lie algebra $\fg$ to have our Property PC. So why do we want to study deformations of finite group schemes $G$ such that every prime is noble for $G$?
The answer is a kind of local-to-global problem for tensor product structures on modules supported at only one point. 


In place of localization we consider the pullback of a  module along a \emph{$\pi$-point} for $G$, as defined by Friedlander and Pevtsova \cite{FrPev07}, \cite{FrPev05}, which is a flat map \[\alpha : k[t] / t^p \to kG\] that factors through a unipotent subgroup scheme of $G.$
When the pullback of a module $M$ to $k[t]/t^p$ modules along $\alpha$ is not projective, we say $M$ is supported by $\alpha$. It is then shown that
the cohomological support for a module $M$ over a finite cocommutative Hopf algebra $A$ is equivalent (see Sections \ref{pppp}, \ref{ttGeoSection}) to the locus of $\pi$-points $\alpha : k[t]/t^p \to A$ such that the pullback is not projective. Our methods for proving Property PC are partly an investigation into whether isomorphisms
\[(M\otimes N)\downarrow_{\alpha} \cong (M\otimes' N )\downarrow_{\alpha},\]
known to hold `locally' for each $\pi$-point $\alpha$, are enough to conclude the `global' isomorphism $M\otimes N \cong M\otimes ' N$ for two different tensor products $\otimes, \otimes'$ of modules over the algebra $A.$ Further, the definition of noble (\ref{nobleDef}) gives us a representing $\pi$-point $\alpha : k[t] / t^p \to A$ such that the restriction property holds
\[(M\otimes N)\downarrow_{\alpha} \cong M\downarrow_{\alpha} \otimes N\downarrow_{\alpha},\] since $\otimes$ may be defined on $k[t]/t^p$-modules according to a Hopf-subalgebra of $A$. Suppose $M, N$ are $A$-modules supported only at the $\pi$-point $\alpha$, which is noble for both finite group schemes $G, G'$ having an isomorphism $kG \cong kG'$ as associative algebras. Then assuming further that $kG$ is a local ring we have an automatic local-isomorphism of products because for the point $\alpha$ in the support of both $M, N$, we get 

\begin{flalign}\nonumber
    (M\otimes N)\downarrow_{\alpha} &= (M \downarrow_{\alpha})\otimes (N\downarrow_{\alpha})\\\nonumber
    &\cong(M \downarrow_{\alpha})\otimes (N\downarrow_{\alpha})\\\label{localIsoProp}
    &\cong (M \downarrow_{\alpha})\otimes' (N\downarrow_{\alpha})\\\nonumber
    &= (M\otimes' M)\downarrow_{\alpha},
\end{flalign}
and for the points not supported by $M, N$ the tensor product is projective on both sides, and hence free of the same rank. The third identity \ref{localIsoProp} is a direct application of our fundamental example \ref{introGa1}. Without automatic local-isomorphisms as a starting point, it is significantly more difficult to address how the tensor categories compare between two arbitrary group schemes $G, G'$ sharing a group algebra. A finite group scheme for which every prime, or $\pi$-point, is noble, is as such a good starting point for this kind of investigation, so Proposition \ref{tautnoble} offers an especially convenient start towards investigating Lie algebras.  
\section{Background}\label{background}
\subsection{Tensor product of representations and reconstruction}\label{reconstructionSection}
If $V, W$ are two representations of a Lie algebra $\fg$ over a field $k$, the tensor product $V\otimes_k W$ is given the structure of a $\fg$-representation, we'll call $V \widetilde\otimes W$, according to the Leibniz rule on simple tensors
\[x(v\otimes w) = xv\otimes w + v\otimes xw,\]
for $x \in \fg, v \in V, w\in W.$ If $k$ is of characteristic $p > 0$ and $\fg$ is a restricted Lie algebra of dimension $r$ in the sense of Jacobson \cite{Jac41}, then the restricted universal enveloping algebra $A = u(\fg)$ is defined, and is an associative algebra over $k$ of dimension $p^r$, such that the restricted representations of $\fg$ are equivalent to modules over $A$. In general, the products of representations, which endow the category of finite representations $\rep\fg$ with the structure of a tensor category, fibred via the usual forgetful functor
\[\cF : \rep \fg \to \cVec k,\]
are all induced from a Hopf algebra structure $A \to A \otimes A$, a map of $k$-algebras. The product $\widetilde \otimes$ comes from the canonical Hopf algebra $\widetilde \Delta : A \to A \otimes A,$ which makes $\fg$ into the primitive subspace of $P(A) \subset A$, i.e. each element $x$ of the generating set $\fg \subset A$ is mapped to $x \otimes 1 + 1 \otimes x \in A \otimes A.$ The monoidal unit $k$ is given module structure by a counit map of $k$-algebras $A \to k$. 

We recall from Etingof, Gelaki, Nikshych, and Ostrik \cite{EGNO15} some definitions and the \emph{reconstruction theorem} \ref{reconstruction} to justify how tensor category structures on $\rep \fg$ are induced from Hopf algebra structures on $A.$
\begin{definition}\cite{EGNO15}
\begin{enumerate}
    \item A $k$-linear abelian category is \emph{locally finite} if
    \begin{enumerate}[(i)]
        \item $\cC$ has finite dimensional spaces of morphisms;
        \item every object of $\cC$ has finite length,
    \end{enumerate}
and $\cC$ is \emph{finite} if in addition
    \begin{enumerate}[resume*]
        \item $\cC$ has enough projectives; and
        \item there are finitely many isomorphism classes of simple objects. 
    \end{enumerate}
    \item An object in a monoidal category is called \emph{rigid} if it has left and right duals. A monoidal category $\cC$ is called \emph{rigid} if every object of $\cC$ is rigid.
    \item Let $\cC$ be a locally finite $k$-linear abelian rigid monoidal category. The category $\cC$ is a \emph{tensor category} if the bifunctor $\otimes : \cC \times \cC \to \cC$ is bilinear on morphisms, and $\End_\cC({\bf 1}) \cong k$. 
\end{enumerate}
\end{definition}
\begin{definition}\cite{EGNO15}
Let $\cC, \cD$ be two locally finite abelian categories over $k$. \emph{Deligne's tensor product} $\cC \boxtimes \cD$ is an abelian $k$-linear category which is universal for the functor assigning, to every $k$-linear abelian category $\cA$, the category of right exact in both variables bilinear bifunctors $\cC \times \cD \to \cA.$ 

The tensor product exists, is locally finite, and is unique up to unique equivalence, see \cite[Proposition~1.11.2]{EGNO15}, or Deligne \cite{Del90}.
\end{definition}

\begin{expo}
    A finite $k$-linear abelian category $\cC$ is equivalent to the category of modules over a finite algebra, or rather a Morita equivalence class of algebras \cite[Section~1.8]{EGNO15}. Fixing an exact faithful functor $\cF : \cC \to \cVec k$ to finite dimensional vector spaces allows a finite algebra to be constructed as $\End(\cF)$. Given an algebra $A$ such that $\cC = \mod A,$ the forgetful functor $\cF$ is representable by the free module $A$. Hence, by the Yoneda lemma $\End(\cF) = \cF(A) = A$ as a vector space, and indeed as an algebra.
    In fact any such exact faithful $\cF$ has that $\cC$ is equivalent to modules over $A = \End(\cF)$, and when modeled as such $\cF : \mod A \to \cVec k$ is isomorphic to the forgetful functor. 

    Let $\cF : \cC \to \cVec k$ and $\cG : \cD \to \cVec k$ be exact faithful functors with finite $k$-linear abelian sources $\cC, \cD,$ equivalent respectively to modules over $A = \End(\cF)$ and $B = \End(\cG)$. Then the exact functors $\cC \to \cD$ relative to $\cVec k$ correspond to homomorphisms of algebras $B \to A.$ In other words, given a commutative diagram of exact functors
\[\begin{tikzcd}
	\cC & \cD \\
	& {\cVec k}
	\arrow["\Phi", from=1-1, to=1-2]
	\arrow["\cF"', from=1-1, to=2-2]
	\arrow["\cG", from=1-2, to=2-2]
\end{tikzcd}\]
the homomorphism of algebras $\Psi : B \to A$, defined by precomposition with $\Phi$, is such that $\Phi$ is the pullback of modules along $\Psi.$ 
\end{expo}

\begin{expo}\label{reconstruction}
Now we review the relevant version of Tannakian reconstruction of a Hopf algebra from a finite tensor category. Let $A$ be a finite associative algebra over $k$, the category of finite modules is a finite $k$-linear abelian category we'll call $\cC.$ It is known that $\cC \boxtimes \cC$ is equivalent to the category of $(A \otimes A)$-modules. By the previous discussion we see the tensor product of finite $k$-linear abelian categories is again finite.

Denote by 
\[\cF : \cC \to \cVec k\]
the forgetful functor. Then the composition
\[\otimes \circ (\cF \boxtimes \cF) : \cC \boxtimes \cC \to \cVec k\boxtimes \cVec k\to \cVec k\]
is equivalent, when $\cC \boxtimes \cC$ is modeled as $\mod (A \otimes A),$ to the forgetful functor.
A tensor category structure on $\cC$ making $\cF$ into a functor of tensor categories is essentially one giving an $A$-module structure to the $k$-space $V\otimes_k W$ for each pair of modules $V, W$. 
In general a tensor category structure on $\cC$ is realizable as an exact functor $\otimes : \cC \boxtimes \cC \to \cC,$
relative to $\cVec k$, together with coherence laws of associators, etc. 
By the previous discussion, these are homomorphisms $A \to A \otimes A$ of algebras with coherence laws of associators corresponding to being a bialgebra. Tensor categories being rigid by definition will in turn see that these bialgebras are indeed Hopf algebras. Thus, we have proved the reconstruction theorem for tensor category structures on a fixed category of modules. That is, the classification of Hopf algebra structures on a given finite associative algebra is equivalent to the classification of tensor category structures on its category of modules. It also follows from reconstruction that Hopf algebra structures on a fixed finite augmented algebra are equivalent to tensor category structures with a fixed choice of unit object. 
\end{expo}

\subsection{$\pi$-points and a plausibility proposition}\label{pppp}
We will first review the machinery of \emph{$\pi$-points}, defined by Friedlander and Pevtsova \cite{FrPev07}, \cite{FrPev05}, which are used in Definition \ref{nobleDef} to define when a homogeneous prime of cohomology is \emph{noble}, a fundamental notion in Definition \ref{mainProperty} of Property PC. From this machinery we conclude a basic proof for Proposition \ref{tautnoble} stated below, which is otherwise a deeper consequence of earlier work of Friedlander and Parshall \cite{FrPar87}, \cite{FrPar86}, as well as Suslin, Friedlander, and Bendel \cite{SFB97}. This shows plausibility for Property PC for any given restricted Lie algebra. Recall for a finite group scheme $G$ over $k$, that the total cohomology $H^*(G, k) = \Ext^*_G(k, k) = \bigoplus_{i \ge 0}\Ext^i_G(k, k)$, for $k$ the trivial representation, is a graded commutative algebra with the cup product (see e.g. Benson \cite{bensonI}).
\begin{prop}\label{tautnoble}
    Let $\fg$ be a finite dimensional restricted Lie algebra over $k$, and $H^*(\fg, k) = \Ext^*_{\fg}(k, k)$ the cohomology ring. Then every homogeneous prime $\fp \in \Proj H^*(\fg, k)$ is the radical ideal $\sqrt{\ker(\alpha^*)}$ for the composition
    \[\alpha^* : H^*(\fg, k) \xrightarrow{\otimes_kK} H^*(\fg_K, K) \xrightarrow{H^*(\iota)} H^*(\fh, K), \]
    induced by the inclusion $\iota : \fh \to \fg_K$ of some 1-dimensional Lie subalgebra $\fh$, restricted by $\fh^{[p]} = 0,$ of the base change $\fg_K = \fg\otimes_k K$ to a field extension $K / k.$
\end{prop}

If $A$ is any algebra over $k$ and $M$ is any module over $A$, we write $A_K = A\otimes_k K$ to be the base changed algebra over $K$ and $M_K = M\otimes_k K$ to be the base changed module over $A_K$. 

\begin{definition}\label{defpipt}\cite{FrPev07}
    Let $G$ be a finite group scheme over $k$. A $\pi$-point of $G$ (defined over a field extension $K / k$) is a (left) flat map of $K$-algebras
    \[\alpha_K : K[t]/t^p \to KG\]
    which factors through the group algebra $KC_K \subset KG_K = KG$ of some unipotent abelian subgroup scheme $C_K$ of $G_K$. 

    If $\beta_L : L[t]/t^p \to LG$ is another $\pi$-point of $G$, then $\alpha_K$ is said to be a \emph{specialization} of $\beta_L$, written $\beta_L \downarrow \alpha_K$, provided that for any finite dimensional $kG$-module $M$, $\alpha^*_K(M_K)$ being free implies $\beta^*_L(M_L)$ is free.

    Two $\pi$-points $\alpha_K, \beta_L,$ are said to be \emph{equivalent}, written $\alpha_K \sim \beta_L$, if $\alpha_K \downarrow \beta_L$ and $\beta_L\downarrow \alpha_K.$
    
    The points $\alpha_K, \beta_L$ are said to be \emph{strongly equivalent} if for any module (not necessarily finite dimensional) $M$, $\alpha^*_K(M_K)$ is projective if and only if $\beta^*_L(M_L)$ is projective. It is shown that equivalence implies strong equivalence, and hence the notions coincide. 
\end{definition}

\begin{expo}\label{QEexpo}
    Denote by $\GG_{a(r)}$ the $r$th Frobenius kernel for the additive group scheme $\GG_a$ over $k$. 
    A \emph{quasi-elementary} group scheme is one in the form \[E \cong \GG_{a(r)} \times (\ZZ / p)^s.\] When $E$ is quasi-elementary, we have isomorphism of the group algebra \[kE \cong k[t_1, \dots, t_{r+s}]/(t_1^p,\dots, t_{r+s}^p), \]
    with the first $r$ variables $t_i$ dual to the basis elements $t^{p^i}$ in the coordinate algebra $k[\GG_{a(r)}] = k[t]/t^{p^r}.$
    It is shown in \cite{FrPev05} that each $\pi$-point defined over the base $k$ (originally called a $p$-point with the same equivalence relation) is equivalent to some $\alpha : k[t]/t^p \to kG$ which factors through the group algebra $kE \subset kG$ for some quasi-elementary subgroup scheme $E\subset G$. In fact the base changed statement can be shown for a $\pi$-point defined over any field extension.


    Notice, for any finite Hopf algebra $A$ with comultiplication $\Delta :A \to A\otimes A$ we may define a restricted Lie subalgebra $P(A)\subset \Lie (A)$ the primitive subspace of $A$, i.e. $x \in A$ such that $\Delta(x)=  x \otimes 1 + 1 \otimes x.$ The universal enveloping algebra $u(P(A))$ is isomorphic to the Hopf subalgebra of $A$ generated as an associative algebra by the subspace $P(A).$
    
    If $E$ is a quasi-elementary group scheme and the group algebra $kE$ is given coordinates $t_i$ as above, direct computation shows that $P(kE)$ is one-dimensional, generated by $t_1.$ Thus, $\GG_{a(1)}$ is the only quasi-elementary group scheme with group algebra isomorphic as a Hopf algebra to a universal enveloping algebra for a restricted Lie algebra. 
 \end{expo}

\begin{expo}\label{piptsCohomology}
Now we review the relationship with cohomology. The algebra $H^*(G, k)$ is graded-commutative, meaning not necessarily commutative. However we do have that every homogeneous element is either central or nilpotent. So we write $\Proj H^*(G, k)$ to mean the space of homogeneous primes for the reduction of $H^*(G, k)$, a commutative, graded algebra. In characteristic $p = 2$ we have that homogeneous elements of any degree may survive, but for characteristic $p > 2,$ this means only the even degree elements may survive.

Denote the algebra $D = K[t]/t^p$ over $K$, an extension field of $k$. There exists a Hopf algebra structure on $D$ (in fact there are two up to isomorphism, by a theorem of Oort and Tate \cite{OT1970}) showing that the Hopf algebra cohomology $H^*(D, K)$ is also a graded-commutative algebra. 
Given a $\pi$-point $\alpha_K : D \to KG_K$ defined over $K / k$, we define the ideal $\fp(\alpha_K)$ as the radical of the kernel for the composition
\[H^*(G, k) \xrightarrow{\otimes_k K} H^*(G_K, K) \xrightarrow{H^*(\alpha_K)} H^*(D, K).\]
It is shown in \cite{FrPev07} that $\fp(\alpha_K)$ is always a homogeneous prime in $ \Proj H^*(G, k),$ that every homogeneous prime of $\Proj H^*(G, k)$ is of the form $\fp(\alpha)$ for some $\pi$-point $\alpha$ of $G$, and further, for two $\pi$-points $\alpha_K, \beta_L,$ that $\alpha_K \sim \beta_L$ coincides with the equivalence relation $\fp(\alpha_K) = \fp(\beta_L).$
\end{expo}

The last thing we need to prove Proposition \ref{tautnoble} is the following theorem of Milnor and Moore, found as \cite[Theorem~6.11]{MM65}
\begin{thm}\label{liesubhopf}
    Let $\fg$ be a finite dimensional restricted Lie algebra over $k$, and $A = u(\fg)$ the restricted enveloping algebra, a cocommutative Hopf algebra with $\fg \subset A$ the subspace of primitive elements. Then if $A' \subset A$ is any Hopf subalgebra, there exists a restricted Lie subalgebra $\fg' \subset \fg$ such that $A' = u(\fg')$ and $A'\subset A$ is the induced inclusion.
\end{thm}

Contrast now Proposition \ref{tautnoble} with the cohomological structure of $\pi$-points reviewed in \ref{piptsCohomology}. Identifying $R = H^*(u(\fg), k) = H^*(\fg, k)$, we claim that each prime in $\Proj R$ comes from a $\pi$-point $K[t]/t^p \to u(\fg_K)$ which is the inclusion of a Hopf subalgebra isomorphic to $K\GG_{a(1)}.$ It is not true in general that $\pi$-points are equivalent to some Hopf subalgebra inclusion, and in fact the failure of this property for general finite group schemes is what we are studying with Property PC (Definition \ref{mainProperty}), at the level of tensor categories. 

\subsubsection*{Proof of Proposition \ref{tautnoble}}
    Let $G$ be the finite group scheme with group algebra $kG = u(\fg)$ as cocommutative Hopf algebras. 
    Let $E_K$ be a quasi-elementary subgroup scheme of $G_K$. By theorem \ref{liesubhopf}, the map $KE_K \to KG_K$ is the induced map of enveloping algebras for a subalgebra of $\fg.$
    In particular, $KE_K$ is generated by its space of primitive elements. Since $E$ is quasi-elementary, by our discussion \ref{QEexpo} we have $E \cong \GG_{a(1)}$ with $KE_K \cong K[t]/t^p.$
    Every $\pi$-point is equivalent to some $\alpha : K[t]/t^p \to KG_K$ factoring through a quasi-elementary subgroup scheme of $E$ of $G$, and in this case we have shown $K[t]/t^p \to KE_K$ is an isomorphism, and $\alpha$ is a map of Hopf algebras assuming $t$ is primitive. 
\qed

\begin{expo}
    The reason we say Proposition \ref{tautnoble} shows plausibility for Property PC is as follows. In Definition \ref{mainProperty}, we have quantified Property PC over all cocommutative Hopf algebra structures on $A = u(\fg)$ for a restricted Lie algebra $\fg.$ One such structure is the Lie comultiplication and associated tensor product $\Delta = \widetilde\Delta, \otimes = \widetilde \otimes,$ defining a finite group scheme we'll call $G$. Without reviewing definitions of $\pi$-support and the categories $\cC(\fp)$ yet, it is tautological that the products $\otimes, \widetilde\otimes$ define the same Green ring structure on $\cC(\fp)$. So for $\fg$ to satisfy PC, it is necessary for each $\pi$-point to be noble for $G$ i.e. equivalent to the inclusion of a Lie subalgebra $\fh \subset \fg$ by Theorem \ref{liesubhopf}. This is what is guaranteed by Proposition \ref{tautnoble}. 
\end{expo}
\begin{expo}\label{generalizedPi}(Generalized $\pi$-points) Let $\alpha_K : K[t]/t^p \to KG_K$ be a flat map. The $\pi$-point condition, that $\alpha_K$ factors through a unipotent subgroup scheme of $G_K$ turns out to be too strong for some of our purposes. We want to argue when the radical of the kernel for the composition $H^*(\alpha_K)\circ (\otimes_kK)$ is a homogeneous prime in $\Proj H^*(G, k)$. Taking $D = K[t]/t^p,$ the cohomology ring is calculated as
\[H^*(D, K) = \Ext_D(K, K) = \begin{cases}
    K[\xi]\otimes \Lambda(\eta)& p > 2\\
    K[\zeta]& p = 2
\end{cases}\]
where $|\xi| = 2, |\eta| = 1$, and $|\zeta| = 1$. In particular, regardless of characteristic, the reduced algebra $H^*(D, K)_{\mathrm{red}}$ is an integral domain. The radical of the ideal $\ker(H^*(\alpha_K)\circ (-\otimes_kK))$ in $H^*(G, k)$ agrees with the kernel of the composition $\mathrm{red}\circ H^*(\alpha_K)\circ (\otimes_kK),$ i.e. changing the target from $H^*(D, K)$ to $H^*(D, K)_{\mathrm{red}}$. This ideal is always a homogeneous prime, and so being in $\Proj H^*(G, k)$ is equivalent to the induced map $\mathrm{red}\circ H^*(\alpha_K)\circ( \otimes_kK)$ being nonzero in some positive degree. Such nondegenerate flat maps $\alpha_K$ we may call \emph{generalized $\pi$-points,} and we denote by  $\fp(\alpha_K)$ the homogeneous prime of $\Proj H^*(G, k)$. Friedlander and Pevtsova \cite{FrPev07} show that $\pi$-points are nondegenerate in this way, i.e. that factoring through a unipotent subgroup scheme of $G_K$ is sufficient for knowing the induced map of cohomology is nonzero in some degree. It remains to be shown, for generalized $\pi$-points $\alpha_K, \beta_L$, that $\fp(\alpha_K) = \fp(\beta_L)$ if and only if $\alpha_K \sim \beta_L$, for an equivalence relation $\sim$ defined similarly in terms of detecting projectivity. However, this can be shown by repeating the methods of \cite{FrPev07} in a straightforward way, so we take it for granted. Note how it immediately follows that generalized $\pi$-points are always equivalent to some $\pi$-point. 
\end{expo}

\begin{lemma}\label{actionOnPiPts}
    Let $G$ be a finite group scheme over $k$, and let $\varphi \in \Aut(kG)$ be an augmented automorphism of the augmented algebra $kG$. Let $\alpha : K[t]/t^p \to KG_K$ be a $\pi$-point of $G$ over $K$. Then the composition $\varphi_K\circ \alpha$ is a generalized $\pi$-point of $G$ over $K$, where $\varphi_K \in \Aut(KG_K)$ is the base change. 
    If $\beta_L$ is another $\pi$-point and $\alpha_K \sim \beta_L,$ then $\varphi_K\circ \alpha_K \sim \varphi_L \circ \beta_L$ as generalized $\pi$-points.  
\end{lemma}
\begin{proof}
    Let $A = KG_K$ be the group algebra over $K$, with $\Delta : A \to A\otimes A$ its cocommutative Hopf algebra structure. 
    Suppose $\alpha_K : K[t]/t^p \to A$ factors through the inclusion $KU \hookrightarrow A$ of a unipotent subgroup scheme $U < G_K.$ The composition $\varphi_K\circ \alpha_K$ is automatically a flat map (since $\alpha_K$ and $\varphi_K$ are both flat). So, to show that $\varphi_K \circ \alpha_K$ is a generalized $\pi$-point, it suffices to show that there is some cocommutative Hopf algebra structure $\Delta'$ on $A$ such that $\varphi_K \circ \alpha_K$ factors through the inclusion $B \hookrightarrow A$ of a local Hopf-subalgebra $B$. Letting $G' = (\Spec A^*, \Delta')$ and $U' = \Spec B^*$, we see $U'$ is a unipotent subgroup scheme of $G'.$ 
    Since the cohomology rings and their induced maps are invariant between Hopf algebra structures, we see any such choice of $\Delta'$ and unipotent subgroup scheme $U' < G'$ will have that the induced map of cohomology from $\varphi_K \circ \alpha_K$ is nondegenerate, making $\varphi_K \circ \alpha_K$ a generalized $\pi$-point of $G$, as shown in \cite{FrPev07}. 
    
    Define $\Delta' = (\varphi_K \otimes \varphi_K) \circ \Delta \circ \varphi_K^{-1}$ (c.f. \ref{twisting}). We see $\Delta'$ is indeed a cocommutative comultiplication for a Hopf algebra structure on $A$. Further, if we define $B = \varphi_K(KU)$ the image of the group algebra, we see $B$ is a Hopf-subalgebra of $A, \Delta'$. Now $B$ is local as it is isomorphic to $KU$, and $\varphi_K \circ \alpha_K$ factors through the inclusion $B \hookrightarrow A.$ We conclude $\varphi_K \circ \alpha_K$ is a $\pi$-point for $G'$ and hence a generalized $\pi$-point for $G$ over $K$. 

    Now let $\beta_L$ be a $\pi$-point over $L / k$ such that $\alpha_K \sim \beta_L$. Then for any $kG$ module $M$ we have $\alpha_K^*(M_K)$ is projective if and only if $\beta_L^*(M_L)$ is projective.
    There is natural isomorphism of $K[t]/t^p$ modules $(\varphi_K \circ \alpha_K)^*(M) \cong \alpha_K^*(\varphi^*(M)_K),$ and similarly for $\beta_L.$ Since $\varphi^*(M)$ is a $kG$ module, we have $\alpha^*_K(\varphi^*(M)_K)$ is projective if and only if $\beta_L^*(\varphi^*(M)_L)$ is projective, and hence $\varphi_K \circ \alpha_K \sim \varphi_L \circ \beta_L.$
\end{proof}

There is a hence a well defined $\Aut(kG)$-action on $\Proj H^*(G, k)$ by
\[\varphi\cdot \fp =  \fp(\varphi \circ \alpha) \text{ for a } \pi\text{-point }\alpha\text{ such that } \fp = \fp(\alpha).\]

\subsection{Cohomological support and tt-geometry}\label{ttGeoSection}
\begin{expo}\label{pisuppExpo}
    In \ref{pppp} we reviewed how the space of $\pi$-points for a finite group scheme $G$ over $k$ is equivalent to $\Xscr(G) = \Proj H^*(G, k)$. Let $R = H^*(G, k)$, a graded-commutative algebra. For each representation $M$ of $G$, the cohomology \[H^*(G, M) = \Ext^*_G(k, M)\] has a canonical graded right-module structure over $R$ by the Yoneda splice product. 
    By the theorem of Friedlander and Suslin \cite{FS97}, $\Xscr(G)$ is a (possibly reducible) projective variety for finite group schemes $G$, and $H^*(G, M)$ gives a coherent sheaf over $\Xscr(G)$ for finite dimensional $M$. To $M$ we therefore associate a closed subvariety of \emph{cohomological support}
    \[\Xscr(G, M) = V(\Ann_R(H^*(G, M))).\]
    It is further shown by Friedlander and Pevtsova \cite{FrPev07} that this subvariety is equivalent to $\pi$-support, i.e. 
    \begin{flalign*}
        \Xscr(G, M) &= \supp_G(M)\\
        &= \left\{ [\alpha_K]\ \mid\ \alpha^*_K(M_K)\text{ is not projective}\right\},
    \end{flalign*}
    defined to range over equivalence classes of $\pi$-points $\alpha_K.$
\end{expo}

\begin{expo}\label{noteOnPiEquivalence}
    Now that we have a definition for support, we can elaborate on the equivalence relation for $\pi$-points given in Definition \ref{defpipt}. We continue our assumption that $k$ is algebraically closed, and so the closed points of $\Xscr(G)$ are all of the form $\fp(\alpha)$ for a $\pi$-point $\alpha$ of $G$, defined over the ground field $k$. The equivalence relation $\alpha \sim \beta$ between $\pi$-points $\alpha, \beta$ of $G$ over $k$ is, by definition, that for any finite module $M$, $\alpha^*(M)$ is projective if and only if $\beta^*(M)$ is projective. In practice, we may be given a finite group algebra $kG$ with generators and relations, and the $\pi$-points of $G$ over $k$ make an affine-algebraic subset $\Pi$ of $\AA(kG) = \Spec S(\cO(G))$. So it is preferable to characterize the equivalence relation in coordinates. It turns out fixing a closed point $\fp \in \Xscr(G)$, in many important cases, there are standard techniques for producing a finite module $M$ such that $\Xscr(G, M) = \{\fp\}.$ When such $M$ is known, the equivalence class of $\pi$-points over $k$ $\{ \alpha \in \Pi \mid \fp = \fp(\alpha)\}$ is the same as
    \[\{\alpha \in \Pi \mid \alpha^*(M)\text{ is not projective}\}.\]
    By computing Jordan canonical forms, fixing $\alpha \in \Pi$ such that $\fp = \fp(\alpha),$ it becomes straightforward to characterize $\beta \in \Pi$ such that $\alpha \sim \beta.$ The computational advantage here is that we only need to consider a single finite module $M$ rather than range over all $M$.
\end{expo}

\begin{expo}
    A universal approach toward support is given for tensor-triangulated categories in Balmer's tensor-triangular geometry \cite{Balmer05}. For representations of a finite group scheme $G$, we look at the tensor triangulated category given by finite stable representations, denoted $\stmod kG.$ That is, with $kG$ the cocommutative group algebra associated to $G$, we look at the category of finite dimensional modules, with Hom spaces between objects $X, Y$ given by 
    \[\uHom_G(X, Y) := \Hom_G(X, Y) / \cP(X, Y), \]
    where $\cP(X, Y)$ is the subspace of maps factoring through a projective module. This category is tensor-triangulated, a general fact for the stable category of a Frobenius category with exact monoidal product (see e.g. Keller, \cite{Keller94}). 
    The work of Benson, Carlson, Rickard \cite{BCR96}, and Friedlander and Pevstova \cite{BIKP15}, \cite{BIKP18}, classifies the thick $\otimes$-ideals of $\stmod kG.$ In Balmer's tt-geometric terms, what this means is that the projective variety $\Xscr(G)$ defined above is the spectrum of the tt-category $\stmod kG$.

     Recall the basic elements for tt-structure on $\stmod kG$: The algebra $kG$ is a cocommutative Hopf algebra with counit $kG \to k$ defining $k$ to be the trivial module, the monoidal unit with respect to the product $\otimes$ of representations (see Section \ref{reconstructionSection}).
     A finite representation $P$ is projective if and only if it is isomorphic to $0$ as an object of $\stmod kG$, and the projective modules form an $\otimes$-ideal, meaning the monoidal structure of $\mod kG$ descends to the quotient $\stmod kG.$ 
     
     The triangulated structure has the suspension autoequivalence $\Sigma,$ defined on objects as (co)syzygies, i.e. $\Sigma M = \coker(\iota)$ 
     where $\iota : M \hookrightarrow I $ 
     is a minimal injective embedding of $M$, and the inverse supspension defined by syzygies $\Sigma^{-1} M = \Omega M = \ker(\epsilon)$ where 
     $\epsilon : P \twoheadrightarrow M$ is a minimal projective cover of $M$. The exact triangles come from exact sequences of modules; see, e.g., Happel \cite{Happel87} or Keller \cite{Keller94}. 
\end{expo}

\begin{definition}
    A full subcategory $\cC$ of $\mod kG$ is \emph{triangulated} if
    \begin{enumerate}[(T1)]
        \item Every finite projective $kG$-module is contained in $\cC$, and
        \item For every short exact sequence of finite modules 
        \[0 \to M' \to M \to M'' \to 0,\]
        if two of the modules in $\{M', M, M''\}$ belong to $\cC$ then so does the third.
    \end{enumerate}
        Notice a triangulated subcategory is closed under isomorphism, and also under the autoequivalences $\Sigma, \Omega.$
        
    The subcategory $\cC$ is called \emph{thick} if in addition
    \begin{enumerate}[resume*]
        \item Whenever $M \oplus M'$ belongs to $\cC$, so do the summands $M, M'$.
    \end{enumerate}
    
    The subcategory $\cC$ is called \emph{$\otimes$-ideal} if in addition
    \begin{enumerate}[resume*]
        \item Whenever $M$ belongs to $\cC$, $N \otimes M$ belongs to $M$ for any module $N$.
    \end{enumerate}
    
    The subcategory $\cC$ is called \emph{radical} if in addition
    \begin{enumerate}[resume*]
        \item If the $n$-fold product $M^{\otimes n}$ belongs to $\cC$, so does the module $M$. 
    \end{enumerate}
\end{definition}

The properties (T1)-(T5) for the tensor category $\mod kG$ all descend to the quotient $\stmod kG$ to define the corresponding notions \cite{Balmer05} for subcategories of the tt-category. Now we recall what it means to be a classifying support data on the tt-category $\stmod kG.$
\begin{definition}\cite{Balmer05} A \emph{support data} on the tt-category $\stmod kG$ is a pair $(\Xscr, \sigma),$ where $\Xscr$ is a topological space and $\sigma$ is an assignment which associates to any object $M \in \stmod kG$ a \emph{closed} subset $\sigma(M) \subset \Xscr$ subject to the following rules:
\begin{enumerate}[(S1)]
    \item\label{SD1} $\sigma(0) = \emptyset$,
    \item\label{SD2} $\sigma(M \oplus M') = \sigma(M) \cup \sigma(M')$,
    \item\label{SD3} $\sigma(\Sigma M) = \sigma(\Omega M) = \sigma(M)$,
    \item\label{SD4} $\sigma(M) \subset \sigma(M')\cup \sigma(M'')$ for any short exact sequence \[0 \to M' \to M \to M'' \to 0,\]
    \item\label{SD5} $\sigma(M \otimes M') = \sigma(M)\cap \sigma(M').$
\end{enumerate}
A support data $(\Xscr, \sigma)$ for $\stmod kG$ is a \emph{classifying support data} if the following two conditions hold:
\begin{enumerate}[(C1)]
    \item\label{CSD1} The topological space $\Xscr$ is noetherian and any non-empty irreducible closed subset $\Zscr \subset \Xscr$ has a unique generic point: $\exists !\ x \in \Zscr$ with $\overline{\{x \}} = \Zscr$,
    \item\label{CSD2} We have a bijection 
\begin{flalign*}
    \{\text{Thomason subsets }&\Yscr \subset \Xscr\} \xrightarrow{\sim}\\ 
    &\{\text{thick }\otimes\text{-ideals }\cJ \subset \stmod kG\}
\end{flalign*}
defined by $\Yscr\mapsto \{ M \in \stmod kG \mid \sigma(M) \subset \Yscr\}$, with inverse
$\cJ \mapsto \sigma(\cJ) := \bigcup_{M \in \cJ} \sigma(M).$
\end{enumerate}
\end{definition}
The theorem of Friedlander and Suslin \cite{FS97} shows that the purely topological condition \ref{CSD1} holds for $\Xscr = \Xscr(G)$. The work of Friedlander and Pevtsova \cite{FrPev07}, \cite{FrPev05},  includes that taking $\Xscr = \Xscr(G)$ and defining cohomological support $\sigma(M) = \Xscr(G, M)$ makes $(\Xscr, \sigma)$ into a support data for $\stmod G$. With Benson, Iyengar, Krause, and Pevtsova \cite{BIKP15}, \cite{BIKP18}, we have indeed that cohomology gives a classifying support data for $\stmod G.$
As a consequence, $\Xscr(G)$ has a certain universal property for support data (see Balmer's \cite[Theorem~5.2]{Balmer05}) making it the spectrum of the tensor triangulated category $\stmod kG.$

Some elementary considerations show that in fact the bijection from condition \ref{CSD2} preserves inclusion. Thus, we know how to characterize minimal (radical, $\otimes$-ideal) thick subcategories; they are the subcategories $\cC(\fp)$ supported at a singleton closed point $\fp \in \Xscr(G).$ 
Any finite module has support a closed set, so those with a singleton support in particular have support a closed point.


\begin{expo}
    On the constancy of classifying support between different Hopf algebra structures: If $A\to k$ is an augmentation map, we can define a graded algebra structure on the cohomology $H^*(A, k) = \Ext^*(k, k)$ by splicing Yoneda extensions. If $A$ is given a Hopf algebra structure such that $A \to k$ is the counit, one shows the splicing of Yoneda extensions is equivalent to the cup product, which is known to be graded-commutative (see e.g. Benson, \cite{bensonI} in the case of cocommutative Hopf algebras). The same is true of the graded right-module structure on $H^*(A, M) = \Ext^*(k, M)$. 

    What we have now, is that even though applying the tt-geometry methods discussed for $A$-modules depends on the existence of a cocommutative Hopf algebra structure on $A$ giving a symmetric monoidal product to begin with, the variety $\Xscr(A) = \Proj H^*(A, k)$ does not depend on which Hopf algebra structure is chosen, and the supports $\Xscr(A, M)$ of a finite module $M$ are in this way also independent. They all satisfy the tensor product property \ref{SD5}, and in particular the subcategories $\cC(\fp)$ supported at a point $\fp \in \Xscr(A)$ are closed under any symmetric tensor product chosen. 
\end{expo}
We have now concluded our review of terminology used in Definition \ref{mainProperty}. 

\section{A property of some Lie algebras}\label{mainPropertySection}
\subsection{Definitions}\label{defsection}
\begin{definition}\label{nobleDef}
    Let $G$ be a finite group scheme over a field $k$. We say a $\pi$-point is \emph{noble for $G$} if it is equivalent to a map 
    \[\alpha : K[t]/t^p \to KG_K\]
    such that the image $\alpha(K[t]/t^p)$ is a Hopf subalgebra of $KG_K,$ and the point is \emph{ignoble for $G$} otherwise. 
    We may also refer to an equivalence class of a $\pi$-point being noble or ignoble, as well as its realization as a point $\fp\in \Proj H^*(G, k).$
\end{definition}

\begin{example}
    The Klein 4-group $G = \langle h, g \mid h^2 = g^2 = (gh)^2 = 1\rangle$, for $p = 2$, has 3 noble $\pi$-points up to equivalence, corresponding to its 3 cyclic subgroups generated by $j = h, g, hg,$ and defining a $\pi$-point $t \mapsto j - 1.$ We revisit this in Sections \ref{sectionp2lemma}, \ref{sectionKlein}. Every flat map $k[t]/t^p \to kG$ is defined by 
    \[t \mapsto a(h- 1) + b(g-1) + c(h-1)(g-1),\]
    for nonzero vector $(a, b) \in k^2$, and equivalence of $\pi$-points identifies the triples $(a, b, c)\sim (a', b', c')$ if and only if $ab' - a'b = 0$ (this can be seen explicitly after classifying all modules, $G$ being of tame representation type). Thus $\Xscr(G) = \PP^1$, i.e. points are given homogeneous coordinates $[a : b]$. The maps corresponding to the three generators $h, g, gh$ in these coordinates are $[1 : 0], [0 : 1], [1 : 1]$ respectively. For $G$ we have now that $[1 : 0]$ is noble and $[a : 1]$ is noble iff $a \in \FF_2$.
\end{example}

\begin{definition}\label{semirings}
    A full subcategory $\cD$ of a finite tensor category $(\cC, \otimes)$ (over a field $k$) is a \emph{semiring subcategory} if the set of isomorphism classes of objects $\cD$ is closed under direct sum and tensor product. If $\Phi: \cC_1 \to \cC_2$ is an equivalence of ($k$-linear abelian) categories between finite tensor categories $(\cC_i, \otimes_i),$ and $\Phi$ restricts essentially to an equivalence of categories $\cD_1 \to \cD_2$ between semiring subcategories $\cD_i \subset \cC_i$, we write $(\cD_1, \otimes_1) \equiv (\cD_2, \otimes_2)$ to mean an $\Phi$ induces an isomorphism of Green rings, i.e. $\Phi(a \otimes_1 b) \cong \Phi(a)\otimes_2 \Phi(b)$ for each pair of objects $a, b\in \cD_1.$
\end{definition}

\begin{definition}\label{PABC}
    Let $\fg$ be a restricted Lie algebra over $k$, and $A = u(\fg)$ the restricted enveloping algebra. Let $\widetilde \Delta, \widetilde \otimes$ be the Lie comultiplication on $A$ and its associated tensor product of $\fg$ representations. When $G$ is a finite group scheme arising from a Hopf algebra structure on the augmented algebra $A$, we let $\Delta, \otimes$ denote the group comultiplication on $A$ and its associated tensor product of $G$-representations. Denote $\mathcal{C}(\fp)$ the minimal thick subcategory of finite $A$-modules, with support a singleton $\fp \in \Proj H^*(A, k)$.
    
    \begin{enumerate}[A.]\label{mainBidirectional}
        \item\label{mainPropertyA} The algebra $\fg$ is said to satisfy \emph{Property PA} if for any finite group scheme $G$ as above, and any noble $\fp \in \Xscr(A)$ for $G$, that $(\cC(\fp), \otimes) \equiv (\cC(\fp), \widetilde\otimes)$ as semiring subcategories (Definition \ref{semirings}).
        \item\label{mainPropertyB} The algebra $\fg$ is said to satisfy \emph{Property PB} if for any finite group scheme $G$ as above, and any ignoble $\fp \in \Xscr(A)$ for $G$, there exists $V, W \in \cC(\fp)$ such that $V\otimes W$ is not isomorphic to $V\widetilde\otimes W.$
    \end{enumerate}
    It is clear from definitions that Properties PA, PB are converse to one another and that in conjunction they are Property PC of \ref{mainProperty}. 
\end{definition}
\begin{expo}\label{considerthequotient}
We must emphasize the formal meaning of quantifying our properties PA and PB over all group schemes $G$ having a given group algebra $A$. On one hand, $k$ being algebraically closed and $A = u(\fg)$ being finite dimensional, we may na{\"i}vely define an affine variety $\Bscr$ of cocommutative bialgebra structures $\Delta : A\to A \otimes A$ with fixed counit $A \to k$, an affine variety $\Sscr$ of linear maps $S : A \to A$, and a closed subvariety $\Hscr \subset \Bscr \times \Sscr$ of cocommutative Hopf algebras $(\Delta, S)$ with $S$ an antipode for the comutiplication $\Delta$. In fact, antipodes being uniquely determined by comultiplications, the composition  $\Hscr \to \Bscr \times \Sscr \to \Bscr$ is injective on $k$-points. On the other hand the work of X. Wang et. al. \cite{NgW23}, \cite{NWW15}, \cite{NWW16}, \cite{WW14}, \cite{XWang13}, \cite{XWang15} classifies Hopf algebras in small dimension \emph{up to equivalence}. For our purposes, these classifications can provide a computation of the orbit space $\Hscr / \Aut(A)$ (of $k$-points) where $\Aut(A)$ is the group of augmented $k$-algebra automorphisms $\varphi : A \to A$ acting on $\Bscr \times \Sscr$ by 
\[(\Delta ,\, S) \mapsto (\Delta^\varphi,\, S^\varphi) = ((\varphi \otimes \varphi)\circ \Delta \circ \varphi^{-1},\, \varphi \circ S \circ \varphi^{-1} ),\]
making $\Hscr$ an invariant subvariety.

In these terms, what we have is that for the properties P $=$ PA, PB, PC, we defined implicitly an existential P$'(s)$ dependent on a set $s\in \Hscr$ of $k$-points such that Property P is in the form
\[\text{P} := \forall s \in \Hscr,\,\, \left(\text{P}'(s)\text{ holds}\right),\]
per definition. But to make good use the work of X. Wang et. al. while avoiding the enormous computation of $\Hscr$, we must confirm a curtailment of Property P to be valid for $u(\fg)$. That is, we would like to range over the set $t \in \Hscr / \Aut(A)$, realized as a complete set of $\Aut(A)$-orbit representatives in $\Hscr,$ and confirm that
\begin{equation}\label{curtailment}
\left(\forall t \in \Hscr / \Aut(A),\,\,\left( \text{P}'(t)\text{ holds }\right)\right) \implies \text{P}.
\end{equation}
\end{expo}

The curtailments \ref{curtailment} for Properties P $=$ PA, PB, PC are not immediate for a given $\fg$. The curtailments would follow if for example it is known that \[\text{P}'(s) \implies \left(\text{P}'(s^\varphi)\quad \forall \varphi \in \Aut(A) \right)\] for each $s = (\Delta, S) \in \Hscr,$ where $s^\varphi = (\Delta^\varphi, S^\varphi)$. To see why this is not immediate, consider the following lemmas.

\begin{lemma}\label{twisting}(Twisting Hopf algebras)
    Let $\varphi \in \Aut(A)$ be an augmented algebra automorphism, and  $\varphi^* : \Proj H^*(A, k) \to \Proj H^*(A, k)$ the induced automorphism on varieties (\ref{actionOnPiPts}). If $\Delta$ is the comultiplication for a group scheme $G$ with $kG \cong A$, denote the twisted group scheme $G^\varphi$ by the comultiplication
    $\Delta^\varphi = (\varphi \otimes \varphi) \circ \Delta\circ \varphi^{-1}$. Then $\fp \in \Proj H^*(A, k)$ is noble for $G$ iff $\varphi^*(\fp)$ is noble for $G^\varphi.$
\end{lemma}

\begin{expo}\label{isotropies}
Denote by $\Omega(A, \fp)$ the isotropy subgroup of $\Aut(A)$ for $\pi$-point $\fp$, and \[\Omega(A) = \bigcap_{\fp } \Omega(A, \fp)\]
the kernel of $\Aut(A) \to \Aut(\Proj H^*(A, k))$ which takes $\varphi$ to $\varphi^*$ as in Lemma \ref{twisting}.

Suppose that Hopf algebras $I = \Hscr / \Aut(A)$ on a given $A = u(\fg)$ are classified up to equivalence, as $\Delta_i ,$ corresponding to the scheme $G_i$ and product $\otimes_i$ for $i \in I$. Then for any Hopf algebra structure $\Delta : A \to A \otimes A,$ there is some $\varphi \in \Aut(A)$ and $i \in I$ with $\Delta = \Delta_i^\varphi.$ For modules $M$ with action $\pi : A \otimes_k M \to M$, denote $M^\varphi$ the twisted module on the same $k$-space $M$, but with action the composition \[\pi\circ (\varphi^{-1} \otimes \id_M) : A \otimes_k M^{\varphi} \to M^\varphi,\]
i.e. $M^{\varphi}:= \varphi(A)\otimes_{A} M,$ the base change along $\varphi.$
Then a prime $\fp \in \Xscr(A)$ belongs to the support variety $\Xscr(A, M)$ if and only if $\varphi^*(\fp)$ belongs to $\Xscr(A, M^{\varphi}).$
\end{expo}
\begin{lemma}\label{curtailmentLemma}
    Suppose for each closed point $\fp \in \Xscr(A)$, each finite $A$-module $M$ with $\Xscr(A, M) = \{\fp\},$ and each isotropy $\varphi \in \Omega(A, \fp),$ that $M^\varphi \cong M.$
    Then the curtailment \ref{curtailment} holds for Properties P $=$ PA, PB, PC. 

    In words for e.g. P $=$ PA: under the same isotropy hypothesis, if for all $i \in I$, and any noble point $\fp$ for $G_i$, we have $(\cC(\fp), \otimes_i) \equiv (\cC(\fp), \widetilde\otimes)$ as semiring subcategories, then indeed $\fg$ satisfies Property PA. 
\end{lemma}

The proofs of Lemmas \ref{twisting}, \ref{curtailmentLemma} are straightforward. One shows with Lemma \ref{twisting} that the isotropy hypothesis of Lemma \ref{curtailmentLemma} has, as a consequence, that $\text{P}'(s) \implies \text{P}'(s^\varphi)$ for Hopf algebras $s \in \Hscr$, and Properties P $=$ PA, PB, PC. But the isotropy hypothesis is not immediate for restricted enveloping algebras $A = u(\fg),$ a counterexample is given in \ref{nPA}.

We conclude this section with two easy examples of restricted Lie algebras of finite representation type which satisfy Property PC. 
\begin{example}\label{exGa1}
    Let $\fg = \langle x \rangle$ be the one dimensional Lie algebra, with trivial restriction $\fg^{[p]}= 0.$ The restricted enveloping algebra is given by
    \[u(\fg) = A = k[x]/x^p,\]
    and by a theorem of Oort and Tate \cite{OT1970}, there are only two Hopf algebra structures on $A$ up to isomorphism. They are given by
    \begin{enumerate}
        \item $\widetilde\Delta : x \mapsto x \otimes 1 + 1 \otimes x$,
        \item $\Delta : x \mapsto x \otimes 1 + 1 \otimes x + x \otimes x,$
    \end{enumerate}
    with tensor products denoted $\widetilde \otimes, \otimes$ respectively. The structure $(A, \widetilde \Delta) $ is equivalent to the group algebra for $\GG_{a(1)}$, and the structure $(A, \Delta)$ is equivalent to the group algebra for $\ZZ / p.$ The algebra $A$ is a quotient of a PID and it is easy to see how there is exactly one indecomposable module $J_i$ of dimension $i$, for $1 \le i \le p$, with $J_p$ the unique indecomposable projective module. It is known (see e.g. Benson \cite{Benson17}) that $J_i\, \widetilde\otimes\, J_j \cong J_i \otimes J_j$ for any $1 \le i, j \le p.$

    It follows that $\fg$ satisfies Property PC. To elaborate, Proposition \ref{tautnoble} (or better yet, a direct computation of cohomology and cup product) tells us that $\Proj H^*(\fg, k)$ consists of a single point $\fp$, represented by the identity for $\fg$, and with that, each $J_i$ belongs to the unique minimal thick subcategory $\cC(\fp)$. Finally, we may apply Lemma \ref{curtailmentLemma}: in this case $\Omega(A)= \Aut(A)$, and we see for $\varphi \in \Omega(A)$, that $J_i^\varphi$ is indecomposable of dimension $i$, hence isomorphic to $J_i$.
\end{example}
\begin{example}
    Let $\fg = \langle x, y\ \mid\ [x,y] = 0, x^{[p]} = y, y^{[p]} = 0\rangle$ be the two dimensional abelian Lie algebra with nontrivial restriction. The restricted enveloping algebra is given by
    \[u(\fg) = A = k[x]/x^{p^2}.\]
    By a corollary of X. Wang \cite{XWang13}, there are only three Hopf algebra structures on $A$ up to isomorphism. They are given by 
    \begin{enumerate}
        \item $\widetilde \Delta : x \mapsto x \otimes 1 + 1 \otimes x,$
        \item $\Delta_1 : x \mapsto x \otimes 1 + 1 \otimes x + \omega(x^p),$
        \item $\Delta_2 : x \mapsto x \otimes 1 + 1 \otimes x + x\otimes x,$
    \end{enumerate}
    with tensor products $\widetilde\otimes, \otimes_1, \otimes_2$ respectively. The structure $(A, \widetilde\Delta)$ is equivalent to the group algebra for the Frobenius kernel $\WW_{2(1)},$ where $\WW_i$ is the algebraic group of \emph{length $i$ Witt vectors}. The comultiplication $\Delta_1$ depends on a term $\omega(x^p)$, defined as 
    \[\omega(y) = \frac{(y\otimes 1 + 1\otimes y)^p - (y^p\otimes 1 + 1\otimes y^p)}{p},\] a formal division by $p$ in characteristic $p$. The structure $(A, \Delta_1)$ is the group algebra for a certain degree $p$ subgroup $G_2$ of the second Frobenius kernel $\WW_{2(2)}$. Both $G_2$ and $\WW_{2(2)}$ are equal to their own Cartier dual.  The structure $(A, \Delta_2)$ is equivalent to the group algebra for $\ZZ / (p^2).$
    
    As in the Oort-Tate example above, there is only one $\pi$-point for these group schemes, this time represented by the map
    \[k[t]/t^p \xrightarrow{x^p} k[x]/x^{p^2},\] a subgroup inclusion making the point noble for all three group schemes. It is not hard to see that the products $\widetilde\otimes, \otimes_1, \otimes_2$ all give the same Green ring, so that $\fg$ has Property PC, again making use of Lemma \ref{curtailmentLemma}.
\end{example}

\subsection{Abelian Lie algebras of dimension 2}\label{sectionp2lemma}
Throughout this section we let $\fg$ be the abelian Lie algebra of dimension 2 with the trivial restriction $\fg^{[p]} = 0,$ and $A = u(\fg)$ the restricted enveloping algebra which we endow with coordinates
\[A = k[x,y]/(x^p, y^p)\]
and take $\widetilde \Delta$ to be the Hopf algebra comultiplication making $x$ and $y$ as primitive, and $\widetilde \otimes$ the corresponding tensor product of $A$-modules. 

For $p > 2$, $A$ is of wild representation type, and we will show that $\fg$ does not meet the hypothesis of Lemma \ref{curtailmentLemma}, and in fact, $\fg$ does not satisfy Property PA. In the tame case $p=2$, we show in Section \ref{sectionKlein} that Lemma \ref{curtailmentLemma} can be applied directly.

\begin{expo}\label{WangClassification}
The Hopf algebra structures on $A$, up to equivalence, are classified by X. Wang in \cite{XWang13}, and the cocommutative structures are given as follows. 
\begin{enumerate}\setcounter{enumi}{-1}
    \item The Lie algebra $k\GG_{a(1)}^2$
    \begin{flalign*}
        \widetilde\Delta : x &\mapsto x \otimes 1 + 1 \otimes x\\
        y &\mapsto y \otimes 1 + 1 \otimes y,
    \end{flalign*}
    \item The quasi-elementary group algebra $k\GG_{a(2)}$
    \begin{flalign*}
        \Delta_1 : x &\mapsto x \otimes 1 + 1 \otimes x\\
        y &\mapsto y\otimes 1 + 1 \otimes y + \omega(x),
    \end{flalign*}
    \item The group-Lie product $k\left(\GG_{a(1)}\times \ZZ / p\right)$
    \begin{flalign*}
        \Delta_2 : x &\mapsto x \otimes 1 + 1 \otimes x\\
        y &\mapsto y \otimes 1 + 1 \otimes y + y\otimes y,
    \end{flalign*}
    \item The discrete group algebra $k(\ZZ / p)^2$
    \begin{flalign*}
        \Delta_3 : x &\mapsto x \otimes 1 + 1 \otimes x + x \otimes x\\
        y &\mapsto y \otimes 1 + 1 \otimes y + y \otimes y.
    \end{flalign*}
\end{enumerate}
For $\Delta_1$ we have used the notation
\[\omega(x) = \frac{(x \otimes 1 + 1 \otimes x)^p - (x^p\otimes 1 + 1\otimes x^p)}{p},\]
a formal division of binomial coefficients by $p$.
\end{expo}
\begin{expo}
    We calculate the spectrum $\Xscr(\GG_{a(1)}^2) = \Proj H^*(\GG_{a(1)}^2, k)$, applying Lemma \ref{tautnoble}, to be $\PP^1,$ since each linear subspace of $\fg_K$ over an extension of fields $K / k$ is a Lie subalgebra with trivial restriction. Each $\pi$-point $\alpha_K$ is of the form 
    \[K[t]/t^p \to K[x, y]/(x^p, y^p)\]
    with $t \mapsto ax + by + \xi,$
    for $a,b \in K$ not both $0$,
    where $\xi$ is a polynomial in the ideal $(x^2, xy, y^2),$ i.e. a higher order term. From Friedlander and Pevtsova \cite{FrPev07}, if we let $\beta_L$ be another $\pi$-point with $t \mapsto a'x + b'y + \xi'$ in the same form over $L$, then $\alpha_K \sim \beta_L$ if and only if there is a common extension $F$ of $K$ and $L$ such that $[a : b] = [a' : b']$ as $F$-points of the projective scheme $\PP^1.$

    The group schemes corresponding to the four cocommutative Hopf algebras listed in \ref{WangClassification} are as we have claimed in notation: \[\GG_{a(1)}^2,\quad \GG_{a(2)},\quad \GG_{a(1)}\times \ZZ / p,\quad (\ZZ / p)^2.\]
    We already know that each $\pi$-point is noble for the Lie algebra, corresponding to $\GG_{a(1)}^2$. The noble points for the three Group schemes representing the remaining points of $\Hscr / \Aut(A)$ are calculated below. One checks that each of these group schemes has finitely many subgroup schemes and that base changing to any field extension does not change the number of subgroup schemes. 
    \begin{enumerate}
        \item The quasi-elementary group scheme $\GG_{a(2)}$ has only one nontrivial proper subgroup, and it is isomorphic $\GG_{a(1)}$. The inclusion of $\GG_{a(1)}$ gives a noble $\pi$-point $k[t]/t^p \to A$ which maps $t \mapsto x.$ Therefore $[1 : 0] \in \PP^1$ is the only noble point for $\GG_{a(2)}$.
        \item The group schemes $\GG_{a(1)}$ and $\ZZ / p$ are disjunct in the sense that any subgroup of the product $\GG_{a(1)} \times \ZZ / p$ is the natural inclusion of a product $H_1 \times H_2$ for $H_1 \le \GG_{a(1)}$ and $H_2 \le \ZZ / p$. Therefore there are two inclusions $\GG_{a(1)}\times 0 < \GG_{a(1)} \times \ZZ / p$ and $0 \times \ZZ/ p  < \GG_{a(1)} \times \ZZ / p$, which give noble $\pi$-points $k[t] / t^p \to A$, mapping $t \mapsto x$ and $t\mapsto y$ respectively. Therefore $[1 : 0]$ and $[0: 1]$ are the only noble points for the product $\GG_{a(1)} \times \ZZ / p$.
        \item The discrete group scheme $(\ZZ / p)^2$ is a 2-dimensional vector space over the prime field $\ZZ/ p$. Therefore the nontrivial proper subgroups are all isomorphic to 1-dimensional subspaces, i.e. the cyclic subgroups. The inclusion of the cyclic subgroup generated by $(i, j) \in (\ZZ/ p )^2$ gives a noble $\pi$-point $k[t]/t^p \to A$ mapping $t \mapsto (x +1)^i(y+1)^j - 1.$ Computing the linear term then tells us that the of the noble points for the discrete group are precisely those in the form $[ i : j ] \in \PP^1$ for $i, j \in \ZZ / p,$ and up to equivalence there are $p + 1$ of them. 
    \end{enumerate}
\end{expo}

\begin{lemma}\label{PB'}
    Let $I = \{\widetilde\Delta, \Delta_1, \Delta_2, \Delta_3\} = \Hscr / \Aut(A)$, with corresponding schemes $G_i$ and tensor products $\otimes_i$ for $i \in I$. 
    Then for any $G_i$, and any ignoble point $\fp$ for $G_i$, there exists $V, W \in \cC(\fp)$ such that $V \otimes_i W$ is not isomorphic to $V \widetilde \otimes W.$
\end{lemma}

\begin{proof}
Given $\fp = [a : b] \in \PP^1,$ we let $\alpha(\fp)$ be the \emph{canonical} $\pi$-point, a map \[k[t]/t^p\to A = k[x, y]/(x^p, y^p)\] which takes $t \mapsto ax + by.$ We generate a module with support $\{\fp\}$ by inducing up the trivial $k[t]/t^p$ module $k$ up to $A$, i.e. take
\[V(\fp) = A\otimes_{k[t]/t^p} k,\]
where $A$ is a $k[t]/t^p$ algebra via $\alpha(\fp).$
Explicitly the $A$-module structure for $V(\fp)$ is given as a $p$-dimensional space over $k$ such that $s_2 = ax + by$ acts as the $0$-matrix and, for $s_1 = cx + dy$ with $\{s_1, s_2\}$ linearly independent,  $s_1$ acts by the nilpotent $p\times p$ Jordan block
\[J_p = 
\begin{pmatrix} 
0       &1      & 0     &\dots  & 0\\
0       &0      & 1     &\dots   & 0\\
\vdots  &\vdots & \vdots&\ddots & \vdots\\
0       &0      &0      &0      &1\\
0       &0      &0      &0      &0
\end{pmatrix},\]
written in a fixed ordered basis $s_1^{p-1}v, \dots, s_1v, v,$ where $v = 1 \otimes 1 \in A \otimes_{k[t]/t^p} k.$
Then $\fp \in \Xscr(G_i, V(\fp))$, since $t$ annihilates the restricted module $\alpha(\fp)^*(V(\fp))$. If $\fq = [c : d] \in \PP^1$ is a closed point not equivalent to $\fp$, then the restriction $\alpha(\fq)^*(V(\fp))$ is free of rank 1, and hence $\fq \not\in \Xscr(G_i, V(\fp)),$ and the generic point is not in $\Xscr(G_i, V(\fp))$ either. Thus $\Xscr(G_i, V(\fp))$ is the singleton $\{\fp\}.$

Now we compute the products $V(\fp) \otimes_i V(\fp)$. First, each canonical $\pi$-point $\alpha = \alpha(\fq), \, \fq \in \PP^1,$ is the inclusion of a subgroup scheme of $\GG_{a(1)}^2.$ Therefore for any $A$-modules $M, M',$ we have
\[\alpha^* ( M \widetilde\otimes M') \cong \alpha^*(M)\, \widetilde\otimes\, \alpha^*(M'),\]
where the right-hand $\widetilde\otimes$ is the tensor product of $\GG_{a(1)}$ representations as in Example \ref{exGa1}. Therefore $V(\fp) \widetilde\otimes V(\fp)$ is annihilated by $s_2$ and is free of rank $p$ when restricted along $s_1,$ so
\[V(\fp) \widetilde\otimes V(\fp) \cong pV(\fp).\]

By Proposition \ref{tautnoble} there are no ignoble points for $\GG_{a(1)}^2,$ so PB$'(\widetilde\Delta)$ is vacuously true. To show PB$'(\Delta_i)$ for $i = 1, 2, 3,$ we show that $V(\fp)\otimes_i V(\fp)$ is not annihilated by $s_2$, for each ignoble point $\fp$ for $G_i,$ and is therefore not isomorphic to $V(\fp)\widetilde \otimes V(\fp)$.
Notice the point $[1 : 0] \in \PP^1$ is noble for each of $G_1, G_2, G_3.$ Therefore an ignoble point is always in the form $[a : 1]$ for $a \in k,$ so we may always assume $s_1 = x$ and $s_2 = ax + y.$

The elements $s_2 \otimes 1, 1\otimes s_2 \in A \otimes A$ annihilate the $k$-space $V(\fp)\otimes V(\fp)$. Therefore the action of $s_2$ for the representation $V(\fp)\otimes_i V(\fp)$ of $G_i$, given by $\Delta_i(s_2),$ is well defined as an element of $A\otimes A / \Sigma, $ where $\Sigma$ is the ideal $(s_2 \otimes 1, 1\otimes s_2) \triangleleft A\otimes A.$ The elements of the quotient $A \otimes A/ \Sigma$ are cosets, which we write as $z + \Sigma$ for $z \in A \otimes A.$ 



\begin{enumerate}
    \item The ignoble points for $G_1 = \GG_{a(2)}$ are $\fp = [a : 1]$ for any $a \in k.$ We have
    \[\Delta_1 : s_2 \mapsto s_2\otimes 1 + 1 \otimes s_2 + \omega(s_1) \in \omega(s_1) + \Sigma.\]
    We described $\omega(s_1)$ in \ref{WangClassification} using a formal division of binomial coefficients by $p$. This makes $\omega(s_1)$ a sum with coefficients on terms $s_1^n\otimes s_1^m$ for $n+m = p$, and $n, m \ge 1.$ Using the same basis elements $s_1^{p-\ell} v$ from our description of the Jordan block $J_p$, we see \[\omega(s_1) \cdot (s_1^{p-2} v\otimes v) = (s_1 \otimes s_1^{p-1})\cdot(s_1^{p-2} v\otimes v) = s_1^{p-1}v \otimes s_1^{p-1} v,\]
    so $\omega(s_1)$ does not annihilate $V(\fp) \otimes_k V(\fp).$
    \item The ignoble points for $G_2 = \GG_{a(1)}\times \ZZ/ p$ are $[a : 1]$ for $a \neq 0$. We have
    \[\Delta_2 : s_2 \mapsto s_2 \otimes 1 + 1\otimes s_2 + (s_2 - as_1)\otimes(s_2 - as_1) \in a^2 s_1 \otimes s_1 + \Sigma.\]
    Using the same $s_1^{p-\ell} v$ we see
    \[(a^2 s_1 \otimes s_1) \cdot (v \otimes v) = a^2 s_1 v \otimes s_1 v,\]
    and in particular, supposing $a \neq 0,$ we have $a^2s_1\otimes s_1$ does not annihilate $V(\fp)\otimes_k V(\fp).$
    \item The ignoble points for $G_3 = (\ZZ/ p)^2$ are $[a : 1]$ for $a^p - a \neq 0.$ We have
    \begin{flalign*}\Delta_3 : s_2 \mapsto s_2 \otimes 1 + 1\otimes s_2 + as_1\otimes s_1 &+ (s_2 - as_1)\otimes(s_2 - as_1)\\ &\in (a + a^2) s_1 \otimes s_1 + \Sigma.
    \end{flalign*}
    The same considerations as for $G_2$ above shows that, supposing $a + a^2 \neq 0,$ we have $(a + a^2)s_1\otimes s_1$ does not annihilate $V(\fp)\otimes_k V(\fp).$ Note that $a^p - a$ is divisible by $a^2 + a$ in $k[a]$ in any characteristic $p > 0$.
\end{enumerate}
We conclude that if $\fp$ is ignoble for $G_i$, then $V(\fp)\otimes_i V(\fp) \not\cong V(\fp)\, \widetilde\otimes\, V(\fp)$, for each $i = G_1, G_2, G_3 \in \Hscr / \Aut(A),$ since $V(\fp)\, \widetilde\otimes\, V(\fp)$ is annihilated by $s_1$ and $V(\fp)\otimes_i V(\fp)$ is not. Thus PB$'(s)$ holds for $s \in \Hscr / \Aut(A).$
\end{proof}

\begin{thm}\label{nPA}(c.f. Theorem \ref{intro2dimthm})
    If $p > 2,$ then $\fg$ does not satisfy Property PA.
\end{thm}
\begin{proof}

Consider the automorphism $\varphi : A \to A$ of algebras, with inverse $\varphi^{-1}$
\begin{align*}
    \varphi : x &\mapsto x      &        \varphi^{-1} : x &\mapsto x\\
        y & \mapsto y + x^2,    &                   y   &\mapsto y - x^2.
\end{align*}
Now fix $\fp = [0 :1]$. By examining linear terms of $\varphi(ax + by)$ we see that $\varphi \in \Omega(A).$ For this reason $M^\varphi$ has the same support $\fp$. We compute directly that the module $M^\varphi$ is determined by the $p \times p$-matrix below, representing the action of a generic $ax + by$
\[
\begin{pmatrix}
    0       & a     & -b     & 0     &\dots  &0      \\
    0       & 0     & a     & -b     &\dots  &0      \\
    \vdots  & \vdots&\vdots & \ddots&\ddots &\vdots \\
    0       &0      &0      & 0     & a     & -b\\
    0       &0      &0      &0      & 0     & a\\
    0&0&0&0&0&0
\end{pmatrix}
\]
Now letting $\alpha(\fp)$ be the canonical $\pi$-point with $t \mapsto y$, we see $\alpha(\fp)^* ( M^\varphi)$ has a Jordan block of size $N = \frac{p + 1}2$ (we have assumed $p$ is odd). In particular $M^\varphi$ is not isomorphic to $M$ because it is not annihilated by $y$, thus contradicting the hypothesis in Lemma \ref{curtailment}.
The Jordan block of size $N$ is also present in $\alpha(\fp)^*(M^{\varphi^{-1}}).$ 
We shall see that $M \widetilde \otimes M$ is not isomorphic as an $A$-module to $M \widetilde\otimes^{\varphi} M,$ despite $\fp$ being noble for both the group scheme $G$ asssociated to the Lie algebra $\widetilde\Delta$, and its twist $G^\varphi$ associated to $\widetilde\Delta^\varphi = (\varphi \otimes \varphi)\circ \widetilde\Delta\circ \varphi^{-1}$. This proves Theorem \ref{intro2dimthm}, as our claim shows that $\fg$ does not satisfy Property PA.

Our argument is as follows: if $\otimes$ is induced from any $\Delta$, and $\otimes^\varphi$ induced from the twist $\Delta^\varphi$, then we have for $A$-modules $V, W,$ that there is equality of $A$-modules \[V^\varphi \otimes ^\varphi W^\varphi = (V\otimes W)^\varphi.\]
Therefore $M\widetilde\otimes M$ is isomorphic to $M \widetilde\otimes^\varphi M = (M^{\varphi^{-1}} \widetilde\otimes M^{\varphi^{-1}})^\varphi$ if and only if $(M\widetilde \otimes M)^{\varphi^{-1}}$ is isomorphic to $M^{\varphi^{-1}} \widetilde\otimes M^{\varphi^{-1}}$. But we see that the latter is false by restricting along $\alpha(\fp).$

We know already that $M\widetilde \otimes M$ is a direct sum of $p$ copies of $M$ from the restriction property of $\widetilde\otimes$ used in \ref{PB'}. In particular $(M\widetilde \otimes M)^{\varphi^{-1}} = p M^{\varphi^{-1}}$ has no Jordan blocks of size $p$ when restricted along $\alpha(\fp)$, the largest Jordan block is instead of size $N = \frac{p + 1}{2}.$

But we also have the restriction property
\[\alpha(\fp)^* ( M^{\varphi^{-1}} \widetilde\otimes M^{\varphi^{-1}}) \cong\alpha(\fp)^* ( M^{\varphi^{-1}})\,\widetilde\otimes \,\alpha(\fp)^* ( M^{\varphi^{-1}}), \]
where the right-hand $\widetilde\otimes$ is of $k[t]/t^p$-modules as in Example \ref{exGa1}. The product $\widetilde\otimes$ of $k[t]/t^p$-modules is known to multiply a pair of Jordan blocks of size $n$ to a module containing a Jordan block of size $\max\{2n - 1, p\}$ (see e.g. Benson \cite{Benson17}). We conclude $\alpha(\fp)^* ( M^{\varphi^{-1}} \widetilde\otimes M^{\varphi^{-1}})$ has a Jordan block of size $2N - 1 = p$ but $\alpha(\fp)^*((M\widetilde \otimes M)^{\varphi^{-1}})$ does not, and hence
\[M\widetilde\otimes M \not\cong M\widetilde\otimes^\varphi M.\]
\end{proof}
\subsection{A tame algebra}\label{sectionKlein}
We continue the assumptions of Section \ref{sectionp2lemma}, and we specialize to $p =2,$ so that $A = k[x, y]/(x^2, y^2)$ is of tame representation type. Note that in X. Wang's classification \ref{WangClassification}, we can now replace the term $\omega(t) = t\otimes t$ for $t \in A$, per its definition. 

The Hopf algebra $\Delta_3$ of \ref{WangClassification} corresponds to the discrete group $G_3 = (\ZZ / 2)^2$. Finite indecomposable $A$-modules were first classified by Ba{\v s}ev \cite{Basev61}, identifying $A = k(\ZZ/2)^2$.  The semirings relative to $G_3, \otimes_3$ for thick subcategories $\cC(\fp)$, supported at noble points $\fp\in \Xscr(A)$ for $G_3$, were also successfully calculated  by Ba{\v s}ev, and we will see that $\fg$ satisfies property PA. So for each $V, W$ supported only at a noble point for $G_3$, we will see that $V \otimes_3 W \cong V\widetilde\otimes W$ with the same methods as likely used by Ba{\v s}ev. For the 2-dimensional modules $V(\fp) = A\otimes_{k[t]/t^2} k$ defined in \ref{PB'}, our computations of $V(\fp)\otimes_3 V(\fp)$ also agree with Ba{\v s}ev for $\fp$ both noble and ignoble points for $G_3$. Note that the complete semiring $(\cC(\fp), \otimes_3)$ for the ignoble points $\fp$ for $G_3$ was initially computed in error in \cite{Basev61}, and corrected first by Conlon in \cite{Conlon65}.   

\begin{expo}
Our argument for Theorem \ref{introKlein} is as follows. The algebra $A$ is local so there is up to isomorphism only one projective indecomposable we call $P$, of dimension $4$. For each closed point $\fp \in \Xscr(A, k) = \PP^1$, the minimal thick subcategory $\cC(\fp)$ contains up to isomorphism only $P$, and for each $n = 1, 2, \dots$, a single indecomposable module $V_{2n}(\fp)$ of dimension $2n$ with support $\{\fp\}$. This is shown by Ba{\v s}ev \cite{Basev61}. We will see for $n =1$ that $V_2(\fp)$ agrees with our induced module $V(\fp)$ defined in \ref{PB'}. Ba{\v s}ev's classification shows for us that the algebra $A$ satisfies the isotropy hypothesis for applying the curtailment Lemma \ref{curtailmentLemma}. By Lemma \ref{PB'} we then know that $\fg$ satisfies Property PB. What remains is to repeat the methods of \cite{Basev61} to conclude that
\begin{equation}\label{basevRule}
V_{2n}(\fp) \otimes V_{2m}(\fp) \cong 2 V_{2\min(n, m)} + (nm - \min(n, m))P
\end{equation}
for each point $\fp$ which is noble for the group scheme $G$, having induced product $\otimes = \widetilde\otimes, \otimes_1, \otimes_2, \otimes_3$, and each $n, m.$ Then by the curtailment Lemma \ref{curtailmentLemma}, as well as Propositon \ref{tautnoble} and X. Wang's classification \ref{WangClassification}, we know in particular that $V_{2n}(\fp) \otimes V_{2m}(\fp) \cong V_{2n}(\fp)\, \widetilde\otimes\, V_{2m}(\fp)$ for each Hopf algebra structure of $\Hscr$, with group scheme $G$, product $\otimes$, and $\fp$ noble for $G$. Since these are all possible pairs of nonprojective indecomposables, we conclude that $\fg$ satisfies Property PA and thus Property PC.
\end{expo}
The next three results below extend the techniques of \cite{Basev61} and reduce Ba{\v s}ev's formula \ref{basevRule} to direct computation with matrices. 

\begin{prop}\label{kleinProjComp}
     Let $n,m \in \ZZ_{> 0}$ and $\fp \in \PP^1.$ Then for $\otimes = \widetilde\otimes, \otimes_1, \otimes_2, \otimes_3,$ we have 
    \[V_{2n}(\fp)\otimes V_{2m}(\fp) = V \oplus (nm - \min(n, m))P\]
    for some finite $A$-module $V$ with no nontrivial projective submodule. 
\end{prop}

\begin{lemma}\label{basevLemma}\cite{Basev61}
        Let $R$ be a (nonunital) associative, commutative ring with a $\ZZ$-linear basis $\{x_s \mid s = 1, 2, \dots\}$ ; $x_0 = 0$. Assume that each product $x_mx_n$ is a nonnegative integer combination of the $x_s$ for $s > 0$, and further that $x_1x_2 = 2x_1$. 
    If the $\ZZ$-linear functional $f : R \to \ZZ$ defined by $f(x_s) = s$ satisfies $f(x_sx_t) = 2\min(s, t)$, then $x_mx_n = 2x_{\min(m, n)}$ whenever $m \neq n.$
\end{lemma}

\begin{cor}\label{caseNM}
    Let $n > m \in \ZZ_{> 0}$ and $\fp \in \PP^1.$ Then for $\otimes = \widetilde\otimes,  \otimes_1, \otimes_2, \otimes_3,$ we have
    \[V_{2n}(\fp)\otimes V_{2m}(\fp) = 2V_{2m}(\fp)\oplus (nm - m)P.\]
\end{cor}
The proof of Proposition \ref{kleinProjComp} comes from the fact that $P = A$ is an injective module over $A$, and has a linear basis $1 , x, y, xy.$ So the element $xy$ annihilates any indecomposable module which is not free. It follows that for the module $M = V_{2n}(\fp)\otimes V_{2m}(\fp)$, the rank of the matrix in $\End_k(M)$ representing $xy \in A$ is equal to the rank of the largest free submodule (a direct summand) of $M$. Then since the ideal $(xy\otimes 1, 1\otimes xy) \triangleleft A\otimes A$ is in the $A\otimes A$-annihilator for $M$, we see all the matrices in $\End_k(M)$ represented by the elements
\[\widetilde \Delta(xy),\quad \Delta_1(xy), \quad \Delta_2(xy), \quad \Delta_3(xy)\in A \otimes A\]
are the same as that of $x \otimes y + y \otimes x \in A\otimes A.$ After properly defining the $A$-modules $V_{2n}(\fp)$ for $n =1, 2, \dots,$ it is easily verified that the matrix representing $x\otimes y + y \otimes x$ indeed has rank $(nm - m)$. Lemma \ref{basevLemma} is a tedious exercise in induction. Corollary \ref{caseNM} is immediate from taking $R$ to be the semiring generated by the indecomposables $x_s = V_{2s}$ with products $\otimes$ defined modulo $P$, and applying Lemma \ref{basevLemma}. 

All that is left to verify the formula \ref{basevRule} is the square in each dimension
\begin{equation}\label{squareForm}
    V_{2n}(\fp) \otimes V_{2n}(\fp) = 2V_{2n}(\fp) \oplus (n^2 - n)P
\end{equation}
when $\fp$ is noble for the group scheme $G$ having product $\otimes = \widetilde\otimes, \otimes_1, \otimes_2, \otimes_3.$
For this, before defining the modules $V_{2n}(\fp)$, we state two more tedious results pertaining to Lemma \ref{basevLemma}. 
\begin{prop}\label{basevRing}
    Let $S$ be any subset of the natural numbers such that no two consecutive numbers are elements of $S$. 
    Let $R = \langle x_i\rangle_{i = 1}^\infty$ be the free abelian group, denoting $x_0 =0,$ with multiplication defined by 
    \[x_sx_t = \begin{cases}
        2x_{\min(s, t)} & s\neq t\\
        2x_{s} & s = t \not\in S\\
        x_{s - 1} + x_{s + 1} & s = t \in S,        
    \end{cases}\]
    extended linearly. 
    Then $R$ is an associative, commutative ring admitting a functional $f : R \to \ZZ$ as in Lemma \ref{basevLemma}. 
\end{prop}
\begin{lemma}\label{greenForKlein}
    Let $R$ be a ring with functional $f : R \to \ZZ$ as in Lemma \ref{basevLemma}. Then there is some subset $S$ of the positive integers such that $R$ is the ring defined in Proposition \ref{basevRing}.
\end{lemma}

Now let us define the modules $V_{2n}(\fp)$ in Ba{\v s}ev's classification theorem below. We omit the classification of modules supported everywhere in $\Xscr(A)$, which includes the structure of the syzygy modules $\Omega^n k = \Sigma^{-n} k$ for $n \in \ZZ$, and that these are all such everywhere-supported indecomposable modules up to isomorphism.

\begin{thm}\label{basevClassificationThm}(Ba{\v s}ev, 1961, \cite{Basev61})
    Let $\fp = [a : b] \in \Xscr(A) = \PP^1$, and define $s_2 = ax + by \in A$ and $s_1 = cx + dy$ such that $s_1, s_2$ forms a basis for the subspace $\langle x , y\rangle \subset A.$
    
    Let $M$ denote a vector space of dimension $2n$, with $k$-linear decomposition into \emph{lower and upper} blocks $M = M_\ell \oplus M_{u}$, with $M_\ell, M_u$ each of dimension $n$. We define $V_{2n}(\fp) = M$ to be the $A$-module defined by following matrix representations of the actions of $s_1, s_2 \in A$
    \[s_1 : \begin{pmatrix}0 & I_n\\0 &0\end{pmatrix},\quad s_2 :
    \begin{pmatrix}0 &\mathfrak{N}_n\\0 &0    \end{pmatrix}, \]
    where $I_n$ is the diagonal ones matrix and $\mathfrak{N}_n$ is an upper triangular nilpotent Jordan block of rank $n - 1.$

    Then we have that
    \begin{enumerate}[I.]
        \item The module $V_{2n}(\fp)$ is, up to isomorphism, not dependent on choice of $a, b, c, d,$ such that $\fp = [a : b] \in \PP^1$, and such that $s_1 = cx + dy$ and $s_2$ are linearly independent,
        \item The module $V_{2n}(\fp)$ is indecomposable,
        \item The support $\Xscr(A, V_{2n}(\fp))$ is $\{\fp\} \subset \PP^1,$ and
        \item Any finite indecomposable module $V$ with support $\Xscr(A,V) = \{\fp\}$ is of even dimension $2n$ and is isomorphic to $V_{2n}(\fp)$, for some $n$. 
    \end{enumerate}
\end{thm}

The following Lemma is proven with the exact same method used for computing $V(\fp)\widetilde\otimes V(\fp) \cong 2V(\fp)$ in \ref{PB'}, after noting $V(\fp) = V_2(\fp)$ and applying the curtailment Lemma \ref{curtailmentLemma}.
\begin{lemma}\label{kleintheorembasecase}
    Let $\fp \in \Xscr(A)$ be a noble $\pi$-point for a group scheme $G$ corresponding to a Hopf algebra structure $\Delta \in \Hscr$ on $A$, with $\otimes$ the product of $G$-representations. Then 
$V_2(\fp)\otimes V_2(\fp) \cong 2V_2(\fp).$
\end{lemma}

Now we prove the formula \ref{squareForm}, giving an original proof of the computations of tensor products, first described in the case $G = G_3$ from \ref{WangClassification}, by Ba{\v s}ev \cite{Basev61} and Conlon \cite{Conlon65} without proof. 
\begin{thm}\label{nobleSquareTheorem}
    Let $G$ be the group scheme associated to one of the Hopf algebras $\Hscr / \Aut(A)  = \{\widetilde\Delta, \Delta_1, \Delta_2, \Delta_3\}$ from X. Wang's classification \ref{WangClassification}. Let $\otimes$ be the associated tensor product of representations of $G$. Then for each noble point $\fp$ for $G$ and each $n$, we have an isomorphism 
    \[V_{2n}(\fp)\otimes V_{2n}(\fp) \cong 2V_{2n}(\fp) \oplus (n^2 - n)P\] 
\end{thm}
\begin{proof}
    We will make explicit computations using matrices in each case of \ref{WangClassification} for representatives $G_i \in \Hscr / \Aut(A)$, and promote this to a general formula of $G$ with our curtailment Lemma \ref{curtailmentLemma}. 
    Throughout, we fix a noble $\pi$-point $\fp \in \Xscr(A)$ for $G$ and denote $V_{2n} = V_{2n}(\fp)$.

    Given the basic matrix construction \ref{basevClassificationThm} of the indecomposable reps $V_{2n}$, the choice of basis induces short exact sequences of modules we call the \emph{canonical monos/epis}
    \[
        0 \to V_{2p} \to V_{2(p+q)} \to V_{2q} \to 0.
    \]
    
    Suppose for $n > 1$ the contrary to formula \ref{squareForm}, which by Lemma \ref{greenForKlein} lets us assume
    \[V_{2n} \otimes V_{2n} \cong V_{2(n-1)} \oplus V_{2(n+1)} \oplus (n^2 - n ) P. \]
    But also by Lemma \ref{greenForKlein} we have
    \[V_{2(n+1)} \otimes V_{2(n+1)} \cong 2V_{2(n+1)} \oplus (n^2 + n)P.\]
    Consider the canonical mono $\iota : V_{2n} \to V_{2(n+1)}.$ Then its $\otimes$-square 
    \[\iota\otimes \iota : V_{2n} \otimes V_{2n}\to V_{2(n+1)} \otimes V_{2(n+1)}, \]
    is also a monomorphism, and by the injectivity of $P$, this descends to 
    \[\iota' :  V_{2(n-1)} \oplus V_{2(n+1)} \to 2V_{2(n+1)} \oplus (2n)P.\]
    We will show directly that the restriction $\iota'$ of $\iota\otimes\iota$ indeed has image contained in the nonprojective component of $V_{2(n+1)} \otimes V_{2(n+1)}$, and argue that $\iota'$ is a sum of canonical monos composed with the split embedding. From this it will follow that the cokernel of $\iota \otimes \iota$ is isomorphic to $V_{4} \oplus (2n)P$. But this is a contradiction:
    by the restriction formula (see \ref{ourpurpose}) there is no exact sequence
    \[0 \to V_{2n} \otimes V_{2n}\to V_{2(n+1)} \otimes V_{2(n+1)} \to V_{4} \oplus (2n)P \to 0, \]
    for its restriction to the subgroup representing $\fp$ is an exact sequence of $k[t]/t^2$-modules (see \ref{exGa1})
\[\begin{tikzcd}[cramped, sep=small]
	0 & {4J_1 \oplus 4(n-1)J_2} & {4J_1 \oplus 4(n+1) J_2} & {2J_1\oplus(4n+1)J_2} & 0
	\arrow[from=1-1, to=1-2]
	\arrow[from=1-2, to=1-3]
	\arrow[from=1-3, to=1-4]
	\arrow[from=1-4, to=1-5]
\end{tikzcd}\]
This is a contradiction when $n > 1$ since $J_2$ is projective and injective, and $4(n + 1) < (4n + 1)+4(n-1).$

For direct computation we consider two cases, one in which $\fp = [1 : 0]$, hence we may assume $s_2 = x,$ and $s_1 = y$, and the other, in which $\fp = [a : 1]$ for $a = 0, 1 \in k$ we take $s_2 = ax + y$ and $s_1 = x.$

Suppose $\fp = [1 : 0]$, $s_2 = x,$ $s_1 = y.$ Now we have
\begin{flalign*}
        \widetilde\Delta :  s_1 &\mapsto s_1\otimes 1 + 1\otimes s_1\\
                            s_2 &\mapsto s_2\otimes 1 + 1\otimes s_2,\\
        \Delta_1        :   s_1 &\mapsto s_1 \otimes 1 + 1 \otimes s_1 + s_2\otimes s_2\\
                            s_2 &\mapsto s_2\otimes 1 + 1\otimes s_2,\\
        \Delta_2        :   s_1 &\mapsto s_1 \otimes 1 + 1 \otimes s_1 + s_1 \otimes s_1\\
                            s_2 &\mapsto s_2\otimes 1 + 1\otimes s_2,\\
        \Delta_3        :   s_1 &\mapsto s_1 \otimes 1 + 1 \otimes s_1 + s_1 \otimes s_1\\
                            s_2 &\mapsto s_2\otimes 1 + 1\otimes s_2 + s_2 \otimes s_2.
\end{flalign*}

Now we describe the action of the $\Delta(s_i) \in A \otimes A$ over the $k$-linear decomposition
\[M\otimes M = (M_\ell \otimes M_\ell) \oplus (M_\ell \otimes M_u) \oplus (M_u \otimes M_\ell) \oplus (M_u\otimes M_u)\]
into blocks of size $n^2,$ where $V_{2n} = M = M_\ell \oplus M_u $ as in Theorem \ref{basevClassificationThm}.
There are four actions to check:
\begin{enumerate}\setcounter{enumi}{-1}
    \item $\Delta = \widetilde \Delta$
    \[\hspace{-0.4in}\Delta(s_1) : \begin{pmatrix}
        0 &1 \otimes I_n&I_n \otimes 1&0\\
        0 &0&0&I_n \otimes 1\\
        0 &0&0&1 \otimes I_n\\
        0 &0&0&0
    \end{pmatrix},\quad \Delta(s_2) : \begin{pmatrix}
        0 &1\otimes \mathfrak{N}_n& \mathfrak{N}_n \otimes 1&0\\
        0 &0&0&\mathfrak{N}_n \otimes 1\\
        0 &0&0&1\otimes \mathfrak{N}_n\\
        0 &0&0&0
    \end{pmatrix},\]
    \item $\Delta = \Delta_1$\[\hspace{-0.4in}\Delta(s_1) : 
    \begingroup
    \setlength\arraycolsep{3pt}
    \begin{pmatrix}
        0 &1 \otimes I_n&I_n \otimes 1&\mathfrak{N}_n\otimes \mathfrak{N}_n\\
        0 &0&0&I_n \otimes 1\\
        0 &0&0&1 \otimes I_n\\
        0 &0&0&0
    \end{pmatrix},
    \endgroup\quad
    \Delta(s_2) : \begin{pmatrix}
        0 &1\otimes \mathfrak{N}_n& \mathfrak{N}_n \otimes 1&0\\
        0 &0&0&\mathfrak{N}_n \otimes 1\\
        0 &0&0&1\otimes \mathfrak{N}_n\\
        0 &0&0&0
    \end{pmatrix},\]
    \item $\Delta = \Delta_2$\[\hspace{-0.4in}\Delta(s_1) : \begin{pmatrix}
        0 &1 \otimes I_n&I_n \otimes 1&I_n\otimes I_n\\
        0 &0&0&I_n \otimes 1\\
        0 &0&0&1 \otimes I_n\\
        0 &0&0&0
    \end{pmatrix},\ \Delta(s_2) : \begin{pmatrix}
        0 &1\otimes \mathfrak{N}_n& \mathfrak{N}_n \otimes 1&0\\
        0 &0&0&\mathfrak{N}_n \otimes 1\\
        0 &0&0&1\otimes \mathfrak{N}_n\\
        0 &0&0&0
    \end{pmatrix},\]
    \item $\Delta = \Delta_3$
    \[\hspace{-0.4in}\Delta(s_1) : 
    \begin{pmatrix}
        0 &1 \otimes I_n&I_n \otimes 1&I_n \otimes I_n\\
        0 &0&0&I_n \otimes 1\\
        0 &0&0&1 \otimes I_n\\
        0 &0&0&0
    \end{pmatrix},\ \Delta(s_2) : 
    \begingroup
    \setlength\arraycolsep{3pt}
    \begin{pmatrix}
        0 &1\otimes \mathfrak{N}_n& \mathfrak{N}_n \otimes 1&\mathfrak{N}_n\otimes \mathfrak{N}_n\\
        0 &0&0&\mathfrak{N}_n \otimes 1\\
        0 &0&0&1\otimes \mathfrak{N}_n\\
        0 &0&0&0
    \end{pmatrix}
    \endgroup,\]
\end{enumerate}
with each block entry of the matrices above representing an $n^2\times n^2$ matrix.

The projective component of $M\otimes M$ is actually a subspace which is independent of choice of $\Delta,$ and free of rank $n^2 - n.$ To see this, as in the discussion following Corollary \ref{caseNM}, we see for each choice $\Delta = \widetilde\Delta, \Delta_1, \Delta_2, \Delta_3,$
that $\Delta(s_1s_2)$ has the same matrix representation as that of $s_1 \otimes s_2 + s_2 \otimes s_1,$ given by the block matrix
\[\begin{pmatrix}
    0 & 0 & 0& I_n\otimes \fN_n + \fN_n \otimes I_n\\
    0 & 0 & 0 & 0\\
    0 & 0 & 0 & 0\\
    0 & 0 & 0 & 0
\end{pmatrix}.\]
We can check that this matrix has rank $n^2 - n$ as follows: writing the linear map  \[I_n\otimes \fN_n + \fN_n \otimes I_n \in \Hom_k(M_u\otimes M_u, M_\ell \otimes M_\ell )\]
with respect to the tensor basis, we get the $n^2\times n^2$ block matrix, with each entry an $n\times n$ matrix, as below
\[\begin{pmatrix}
    \fN_n & I_n & 0 &\dots &0\\
    0     &\fN_n&I_n&\dots &0\\
    \vdots&\vdots&\vdots&\ddots&\vdots\\
    0&0&0&\fN_n&I_n\\
    0&0&0&0&\fN_n
\end{pmatrix}.\]
Then of the columns numbered $0, \dots, n^2 - 1,$ the $i$th column is a pivot column if and only if $i$ is not divisible by $n$ (i.e. it is not the leftmost column of its respective block).  
The basis elements $\fB$ of $M_u\otimes M_u$ corresponding to the pivot columns identified in turn generate a free submodule $F^{\fB} \subset M\otimes M$ of rank $n^2 - n$, with total dimension $4(n^2 - n)$ over $k$.

Now, the canonical mono $\iota : V_{2n} \to V_{2(n+1)}$ is induced from an inclusion of basis elements, between upper blocks $V_{2n(u)} \to V_{2(n+1)(u)}$ and lower blocks $V_{2n(\ell)} \to V_{2(n+1)(\ell).}$ Our assertion now follows, that $\iota\otimes \iota : V_{2n} \otimes V_{2n} \to V_{2(n+1)}\otimes V_{2(n+1)}$ is the direct sum of monomorphisms between projective components and between nonprojective components, for each $\otimes = \widetilde\otimes, \otimes_1, \otimes_2, \otimes_3$.

Further, the basic module structure of the indecomposables $V_{2n}$ in Theorem \ref{basevClassificationThm} tells us any embedding $V_{2n} \to V$, for $V$ a finite module with $\Xscr(A, V) = \{\fp\}$, and such that $V$ is annihilated by $s_1s_2$, is a canonical mono into one indecomposable factor of $V$. Thus the assertion follows that $\iota' : V_{2(n-1)} \oplus V_{2(n+1)} \to 2V_{2(n+1)} \to 2V_{2(n+1)} \oplus (2n)P$ is a direct sum of canonical monos. Hence, a contradiction, as explained above. 

We will be able to make the same contradiction by observing the same projective components in the case of $\fp = [a : 1]$ for $a = 0,1 \in k$, but we omit the matrices that let us see this directly. 
\end{proof}

\begin{thm}\label{restatedIntroKlein}(c.f. Theorem \ref{introKlein})
    The restricted Lie algebra $\fg$ satisfies Property PC. 
\end{thm}
\begin{proof}
    From Theorem \ref{basevClassificationThm} and Lemma \ref{twisting}, our curtailment Lemma \ref{curtailmentLemma} applies. Then with Lemma \ref{PB'}, we have that $\fg$ satisfies Property PB. 
    
    We know from Lemma \ref{kleintheorembasecase}, and the calculation of projective components in Proposition \ref{kleinProjComp}, that we may apply \ref{basevLemma} to calculate that each non-square product of indecomposables in $\cC(\fp)$ agrees with Ba{\v s}ev's formula \ref{basevRule}  regardless of $G_i$ chosen from \ref{WangClassification}, and regardless of nobility of $\fp$ for $G_i$. Finally Theorem \ref{nobleSquareTheorem} tells us that the remaining products of indecomposables, the squares of those supported at a noble point, all follow the same formula \ref{basevRule}. Then by the curtailment Lemma \ref{curtailmentLemma}, we have that $\fg$ satisfies property PA.  
\end{proof}

\subsection{Conjectures for Lie algebras}
Consider the Lie algebra $\sl_2$, with presentation
\[\langle e, f, h \mid [e, f] = h, [h, e] = 2e, [h, f] = 2f\rangle.\]
In characteristic $p=2$, any restriction $\fg$ of $\sl_2$ which includes relations \[e^{[2]} = f^{[2]} = 0\] has in turn that $u(\fg)$ is a quotient of the tame noncommutative algebra $B = k\langle x, y\rangle / (x^2, y^2).$ The modules of $B$ are classified by Bondarenko \cite{Bon75}, originally applied to determine how dihedral groups are tame. In turn the work applies directly to the case of the restriction $\fg_1$ of $\sl_2$ which has $h^{[2]} = 0$ (this restricted Lie algebra coincides with the canonical restriction for the Heisenberg algebra). This is because the dihedral group of order $8$ has group algebra isomorphic as an associative algebra to $u(\fg_1)$. Thus, $A = u(\fg_1)$ has been shown so far to have two different cocommutative Hopf algebra structures, but it is an open problem to classify all of the cocommutative Hopf algebra structures on $A$. There are at least ten, including those of the Lie algebra and dihedral group, each specializing under $h\mapsto 0$ to a Hopf algebra on the commutative algebra $k[e,f]/(e^2,f^2)$ covered in Sections \ref{sectionp2lemma} and \ref{sectionKlein}. Similar local algebras of order $p^3$ have Hopf algebra structures classified as a corollary of the work of Nguyen, L. Wang, and X. Wang \cite{NWW15} for $p > 2.$

The other restriction $\fg_2$ on $\sl_2$ with $u(\fg_2)$ a quotient of $B,$ is the canonical restriction derived from the trace-free $2\times2$ matrix representation of $\sl_2$, which has that $h^{[2]} = h.$ We present three conjectures, each encompassed by the next.
\begin{conjecture}\label{conj1}
    The restrictions $\fg_1$ and $\fg_2$ of the Lie algebra $\sl_2$ satisfy Property PC.
\end{conjecture}
\begin{conjecture}\label{conj2}
    Let $\fg$ be a restricted Lie algebra of tame representation type over $k$. Then $\fg$ satisfies Property PC.
\end{conjecture}
\begin{conjecture}\label{conj3}
    Let $\fg$ be any restricted Lie algebra over $k$. Then $\fg$ satisfies Property PC if and only if $\fg$ is of tame or finite representation type. 
\end{conjecture}
A fourth conjecture is also likely to hold, but we also suspect a proof would require classification of Hopf algebras far beyond what is known. 
\begin{conjecture}\label{conj4}
    All restricted Lie algebras satisfy Property PB regardless of representation type. 
\end{conjecture}

The first point to address the `if' direction of Conjecture \ref{conj3} is to show whether tame and finite representation type is equivalent to the isotropy hypothesis of Lemma \ref{curtailmentLemma}. For converse, in the wild case, we can continue in Section \ref{sectionWild} to give ad-hoc arguments for how the failure of the isotropy hypothesis leads to a failure of Property PA, as per our technique in \ref{nPA}, which proved Theorem \ref{intro2dimthm}. Tame restricted Lie algebras of dimension $\le 3$ for odd characteristic may also be studied making direct use of Nguyen, L. Wang, and X. Wang's classification \cite{NWW15}, towards Conjecture \ref{conj2}.

\section{Lie algebras of wild representation type}\label{sectionWild}
In this section we will move toward one direction of Conjecture \ref{conj3}, that no Lie algebra of wild representation type satisfies Property PA. 

We begin by adapting our proof of Theorem \ref{nPA} into a more general situation. Recall the setting for disproving Property PA for a given restricted Lie algebra $\fg$: we wish to produce a Hopf algebra structure $\Delta$ on the enveloping algebra $A = u(\fg)$, differing from the Lie structure $\widetilde \Delta$, corresponding to a tensor structure $\otimes$ on $A$-modules differing from the Lie structure $\widetilde\otimes$. Then we produce $A$-modules
$V, W$, with the support condition $\Xscr(A, V) = \Xscr(A, W) = \{\fp\},$ where $\fp$ is noble for the group scheme $G$ corresponding to $\Delta$, but there is non-isomorphism
\[V\otimes W \not\cong V\widetilde\otimes W.\]

Our technique proving Theorem \ref{nPA} was to leverage the action of $\Aut(A)$ on the space of Hopf algebra structures $\Hscr$, on $A$-modules, and on the variety $\Xscr(A)$, with the natural isomorphism 
\[(V\otimes W)^\varphi \cong V^{\varphi}\otimes^\varphi W^{\varphi}.\]
Specifically, we found one representation $V$ of $\widetilde G = \GG_{a(1)}^2$ satisfying a polynomial identity $\rho^2 = n\rho,$ i.e.
\begin{equation}\label{poly1}
    V\widetilde \otimes V \cong nV,
\end{equation}
where $n = \dim V$, in the Green ring for the infinitesimal group scheme $\widetilde G$, which  corresponds to a Lie Hopf algebra structure $\widetilde\Delta$. But there exists an augmented automorphism $\varphi \in \Aut(A)$ such that the identity $\rho^2 = n\rho$ is not satisfied by $V^{\varphi^{-1}},$ provided $p > 2.$ Now, if we take $\otimes = \widetilde\otimes^\varphi,$ having $V \otimes V \cong V\widetilde\otimes V$ is equivalent to $V^{\varphi^{-1}} \widetilde\otimes V^{\varphi^{-1}}  \cong nV^{\varphi^{-1}}$, which is known to be false. All that is left is to confirm the support condition that $\Xscr(A, V) = \Xscr(A, V^{\varphi^{-1}}) = \{\fp\}$ for some point $\fp \in \Xscr(A).$ This follows from construction, that $V$ was chosen to meet this support condition and $\varphi$ was chosen to be an isotropy in $\Omega(A, \fp) < \Aut(A)$ (\ref{isotropies}).

We will see that very little needs to be changed for a given \emph{abelian} restricted Lie algebra $\fg$ of wild representation type. The polynomial identity \ref{poly1} will always be satisfied by a given induced module $V = k \uparrow^\fg_\fh,$ but not by some twist $V^{\varphi^{-1}},$ for a choice of Lie subalgebra $\fh \subset \fg$ with $k$ the trivial $\fh$-module. In fact, in most cases $\fh$ can be taken to be a subalgebra isomorphic to $\langle t \mid t^{[p]} = 0\rangle.$ For $p =2$, we will see in Proposition \ref{primordialProp2}
that some wild abelian algebras require extra care.

For nonabelian Lie algebras, the polynomial identity \ref{poly1} usually fails for similarly constructed induced modules. But Frobenius reciprocity can still be used to find other polynomial identities involving induced modules: whenever $\fh \subset \fg$ is a Lie subalgebra, denoting $\widetilde\otimes$ the tensor product for both $\fh$ and $\fg$ representations, we have a natural isomorphism
\begin{equation}\label{IndTensor}
M\uparrow_{\fh}^\fg \widetilde \otimes N \cong ( M \widetilde \otimes (N\downarrow^{\fg}_{\fh}))\uparrow^{\fg}_\fh,
\end{equation}
whenever $M$ is an $\fh$-representation and $N$ is a $\fg$-representation. Choosing $\fh$ to be a subalgebra with an easily calculated Green ring can then let us derive ad-hoc polynomial identities (\ref{MackeyPoly}, \ref{polyIds}) for induced representations of $\fg$, which fail after twisting by some isotropy $\varphi \in \Omega(A, \fp)$. We offer the Heisenberg Lie algebra of arbitrary dimension $2n +1, n \ge 1$, as an example of a nonabelian Lie algebra, of wild representation type ($p > 2$), for which this generalized technique can be applied.

\subsection{Induction from the nullcone} 
Let $\fg$ be a restricted Lie algebra over a field $k$ of characteristic $p$. Define $\Nscr_{r}(\fg)$ to be the \textit{$r$th restricted nullcone} of $\fg$, i.e. the subset of $\fg$ defined by
\[\Nscr_{r}(\fg) = \{ x \in \fg \mid x^{[p]^r} = 0\},\]
with $\Nscr_0(\fg) = {\bf 0}$ and $\Nscr(\fg) = \bigcup_{r}\Nscr_{r}(\fg).$ 

Each $\Nscr_{r}(\fg)$ is a homogeneous subvariety (i.e. a cone) of the affine space $\AA(\fg) = \Spec S(\fg^*)$. A simple argument showed in Proposition \ref{tautnoble} how the projective variety $\PP(\Nscr_1(\fg))$ covers the support variety $\Xscr(u(\fg)) = \Proj H^*(\fg, k),$ using the machinery of $\pi$-points. But the earlier work of Friedlander and Parshall \cite{FrPar86} showed how the support variety $\Xscr(u(\fg))$ maps homeomorphically onto $\PP(\Nscr_1(\fg))$, an inverse to our map using $\pi$-points.

We will say the elements in the difference of sets $x \in \Nscr_{r}(\fg)\setminus \Nscr_{r-1}(\fg)$ have \emph{order $r$}. 
When $x \in \Nscr(\fg)$ has order $r$, we denote $\langle x \rangle \subset \fg$ to be the subalgebra of dimension $r$, with basis $x, x^{[p]}, \dots x^{[p]^r}$. Notice the restricted enveloping algebra $u(\langle x\rangle)$ is of the form 
\[k[x]/x^{p^r},\]
a Hopf algebra with $x$ primitive. This Hopf algebra agrees with the group algebra for the infinitesimal group $\WW_{r(1)}$ of length $r$ Witt vectors. When $r = 1$, we have $\WW_{1(1)} \cong \GG_{a(1)}.$ 

For nonzero $x \in \Nscr(\fg)$ having any order $r\ge 1$, we have that $x^{[p]^{r-1}}$ has order 1. Thus the subalgebra $\langle x \rangle \subset \fg$ will always produce a $\pi$-point over $k$ for the infinitesimal group scheme $\widetilde G$ corresponding to $\fg$, in the form of the composition
\[u(\langle x^{[p]^{r-1}}\rangle ) \hookrightarrow u(\langle x \rangle) \hookrightarrow u(\fg).\]
For $x \in \Nscr(\fg)$, we will denote $\fp(x) \in \Xscr(u(\fg))$ the homogeneous prime arising from this $\pi$-point.

\begin{expo}\label{WittVectorExpo}(Representations of Witt vectors)
    Consider the algebra $A = k[x]/x^{p^r}$ over a field $k$ of characteristic $p$. Assuming $x$ is primitive, i.e. $\widetilde\Delta(x) = x\otimes 1 + 1 \otimes x$, defines $A$ to be a cocommutative Hopf algebra, and in fact a restricted enveloping algebra for the $r$-dimensional $P(A) = \langle x\rangle$ ($x$ being of order $r$ in $\Nscr(P(A))$). The infinitesimal group scheme corresponding to $P(A)$ is denoted $\WW_{r(1)}$, the first Frobenius kernel for length $r$ Witt vectors. 

    The representations of $\WW_{r(1)}$ are modules over $A$. The indecomposable representations are thus described by Jordan blocks $J_i$, for $1 \le i \le p^r,$ each of dimension $i$. The block $J_{p^r}$ is the unique indecomposable projective $A$-module. The Green ring for $\WW_{r(1)}$ is easily calculated, and especially well known when $r = 1.$ For now we will only use the following calculation, pertaining to the polynomial identity \ref{poly1}.

    Let $G = \WW_{r(1)}$. The blocks $J_{p^s}$ for $s \le r$ are isomorphic to the induced modules $k\uparrow_H^G$ of trivial modules for the infinitesimal subgroup $H$ corresponding to the subalgebra $\langle x^{p^{r-s}}\rangle$. By Frobenius reciprocity \ref{IndTensor}, we get the identity
    \[J_{p^s} \widetilde\otimes J_{i} \cong J_i\downarrow_{H}^G \uparrow^G_H.\]
    In particular, since $J_{p^s}$ is annihilated by $x^{p^{r-s}}$, we get that each $J_{p^s}$ satisfies the polynomial identity \ref{poly1}, i.e.
    \[J_{p^s}\widetilde\otimes J_{p^s} \cong n J_{p^s},\]
    where $n  = p^s = \dim J_{p^s}.$
    In fact, the only $G$-modules $M$ (of dimension $n$) such that $M\widetilde\otimes M \cong nM$ are of the form $M \cong \frac{n}{p^s}J_{p^s}$ for some $s$. This follows from basic considerations using Jordan canonical forms. 
\end{expo}

\begin{definition}\label{MackeyPoly} For a restricted Lie algebra $\fg$ and $t \in \Nscr(\fg)$ of order $r > 0$, denote $V_{i, \fg}(t) = J_i\uparrow_{\langle t \rangle}^{\fg}$ for $i = 1, \dots, p^r$. Define the \emph{$n^{th}$ Mackey coefficients at $t$} for $\fg$, denoted $c_{i,\fg}^{n}(t)$, such that
\[V_{n, \fg}(t)\downarrow^{\fg}_{\langle t \rangle} \cong \sum c_{i,\fg}^n(t) J_i\]
as $\langle t \rangle$-representations. We define the polynomial identity
\[\rho_1^2 = \sum c_{i,\fg}^1(t) \rho_i,\]
we call the \emph{$1^{st}$ Clebsch-Gordon-Mackey (CGM) polynomial identity} for $\fg$-representations. Hence, by Frobenius reciprocity, the $1^{st}$ CGM polynomial identity is always satisfied by $\rho_i = V_{i, \fg}(t)$, i.e.
\[V_{1, \fg}(t) \widetilde\otimes V_{1, \fg}(t) \cong \sum c_{i,\fg}^1(t)V_{i, \fg}(t).\]
Higher CGM polynomials are derivable from the Green ring for $ \WW_{r(1)},$ but we don't make use of these. 
\end{definition}

\begin{definition}\label{polyIds}
    Let $\fg$ be a restricted Lie algebra. Let \[ F = F(\rho_1, \dots, \rho_n) \in \ZZ [\rho_1, \dots, \rho_n]\] be an integer polynomial, with $F_+, F_-$ the positive and negative components of $F$ respectively, so that $F = F_+ + F_-$. If $V_1, \dots, V_n$ are representations of $\fg$ and $G \in \ZZ[\rho_1, \dots, \rho_n]$ is a polynomial with positive coefficients, we write $G(V_1, \dots, V_n)$ to mean the representation of $\fg$ built from sums $\oplus$ and products $\widetilde\otimes.$
    If $t \in \Nscr(\fg)$ is of positive order and there is isomorphism \[F_+(V_1, \dots, V_n)\downarrow^{\fg}_{\langle t \rangle} \cong -F_-(V_1, \dots, V_n)\downarrow^{\fg}_{\langle t \rangle}\]
    as representations of $\langle t \rangle$,
    then we say \emph{the identity $F = 0$ (or equivalently $F_+ = -F_-$) is witnessed by $t$ for $\rho_i = V_i.$}
    Note that if there exists some $t$ such that a polynomial identity is not witnessed by $t$ for $\rho_i = V_i,$ then the polynomial identity does not hold as representations of $\fg$. 
\end{definition}

The following lemma may be used to extend the negation of Property PA to a larger restricted Lie algebra in a general, nonabelian setting, provided that the induced modules $V_{i, \fg}(t)$ are not projective.  

\begin{lemma}\label{generalAutomorphismExtension}
    Let $\fg'$ be a restricted Lie algebra and $B = u(\fg')$ the restricted enveloping algebra. Assume there exists $t \in \Nscr(\fg')$ with associated prime $\fp' = \fp(t) \in \Xscr(B)$, and an isotropy $\psi \in \Omega(B, \fp')$ such that the $1^{st}$ CGM polynomial identity is not witnessed by $t$ for $\rho_i = V_{i, \fg'}(t)^{\psi^{-1}}$. 

    Now suppose that $\fg$ is a restricted Lie algebra with $\fg' \subset \fg$ a \emph{central} Lie subalgebra, i.e. $[\fg' , \fg] = 0$. Suppose that there exists an augmented automorphism $\varphi$ of the algebra $A = u(\fg)$ extending $\psi,$ i.e. denoting $A = u(\fg),$ that the following diagram commutes
\[\begin{tikzcd}[cramped]
	B & A \\
	B & A.
	\arrow[hook, from=1-1, to=1-2]
	\arrow["\psi"', from=1-1, to=2-1]
	\arrow["\varphi", from=1-2, to=2-2]
	\arrow[hook, from=2-1, to=2-2]
\end{tikzcd}\]
    
    Then the $1^{st}$ CGM polynomial identity is not witnessed by $t$ for $\rho_i = V_{i, \fg}(t)^{\varphi^{-1}},$ where $t \in \Nscr(\fg)$ by the inclusion $\Nscr(\fg') \subset \Nscr(\fg).$
    
\end{lemma}
\begin{proof}
    Let $W_i = V_{i, \fg'}(t) = J_i\uparrow_{\langle t \rangle }^{\fg'}$ and $V_i = V_{i, \fg}(t) = W_i \uparrow_{\fg'}^{\fg}.$ Let $d_i = c^1_{i, \fg'}(t)$ and $c_i =c^1_{i, \fg}$. Now we have
    \[W_1 \downarrow^{\fg'}_{\langle t \rangle} \cong \sum d_i J_i,\qquad V_1 \downarrow^{\fg}_{\langle t \rangle} \cong \sum c_i J_i.\]
    Since $\fg' \subset \fg$ is central, we have $W\uparrow_{\fg'}^{\fg}\downarrow_{\fg'}^{\fg} \cong nW$ for all representations $W$ of $\fg',$ and hence also  $c_i = n d_i,$ where $n = [\fg : \fg'].$ This follows from the theorem of Nichols and Zoeller \cite{NZ89}, which states that $u(\fg)$ is free of rank $n$ as a left module over the Hopf-subalgebra $u(\fg'),$ and the PBW theorem. 

    By assumption, there is non-isomorphism
    \[
        (W_1^{\psi^{-1}}\widetilde \otimes W_1^{\psi^{-1}})\downarrow^{\fg'}_{\langle t \rangle}\not\cong \sum d_i W_i^{\psi^{-1}}\downarrow^{\fg'}_{\langle t \rangle},
    \]
    and we want to show
    \[
        (V_1^{\varphi^{-1}}\widetilde\otimes V_1^{\varphi^{-1}})\downarrow^{\fg}_{\langle t \rangle}\not\cong \sum c_i V_i^{\varphi^{-1}}\downarrow^{\fg}_{\langle t \rangle}.
    \]
    Let $C$ denote the image $\psi(u (\langle t \rangle)) \subset B.$ Since $\varphi$ extends $\psi$, we have also $C = \varphi(u(\langle t \rangle)) \subset A.$ For representations $W$ of $\fg'$ and $V$ of $\fg$, we have by definition $W^{\psi^{-1}}\downarrow^{\fg'}_{\langle t \rangle} = W\downarrow^B_C$ and $V^{\varphi^{-1}}\downarrow^{\fg}_{\langle t \rangle} = V\downarrow^A_C$ as representations of $\WW_{r(1)}$, where $r$ is the order of $t$. 
    Now on one side we have isomorphisms
    \begin{flalign*}
        (V_1^{\varphi^{-1}}\widetilde\otimes V_1^{\varphi^{-1}})\downarrow^{\fg}_{\langle t \rangle} &\cong V_1\downarrow^A_C \widetilde\otimes V_1\downarrow^A_C\\
        &\cong k\uparrow_{\langle t \rangle}^{\fg'}\uparrow^{\fg}_{\fg'}\downarrow^A_{B}\downarrow^{B}_{C} \widetilde\otimes k\uparrow_{\langle t \rangle}^{\fg'}\uparrow^{\fg}_{\fg'}\downarrow^A_{B}\downarrow^{B}_{C}\\
        &\cong n^2 W_1 \downarrow^{B}_{C} \widetilde\otimes W_1 \downarrow^{B}_{C}\\
        &\cong n^2 (W_1^{\psi^{-1}}\widetilde \otimes W_1^{\psi^{-1}})\downarrow^{\fg'}_{\langle t \rangle}.
    \end{flalign*}
    On the other side we have similarly $V_i^{\varphi^{-1}}\downarrow^{\fg}_{\langle t \rangle} \cong n W_i^{\psi^{-1}}\downarrow^{\fg'}_{\langle t\rangle }$ for each $i$, and hence
    \[\sum c_i V_i^{\varphi^{-1}}\downarrow^{\fg}_{\langle t \rangle} \cong n^2 \left( \sum d_i W_i^{\psi^{-1}}\downarrow^{\fg'}_{\langle t \rangle}\right).\]
    The desired non-isomorphism follows immediately. 
\end{proof}

For abelian restricted Lie algebras, or more generally whenever $t \in \Nscr(\fg)$ is central, the Mackey coefficients are quite simple. We state the next few results which are considerably specialized to this situation, and directly adapt the proof of Theorem \ref{nPA}. 

\begin{lemma}\label{nPALemma}
    Let $\fg$ be a restricted Lie algebra and $A = u(\fg)$ its restricted enveloping algebra. Suppose there exists a \emph{central} nilpotent element $t \in \Nscr(\fg)$ with associated prime $\fp = \fp(t) \in \Xscr(A)$, and an isotropy $\varphi \in \Omega(A, \fp)$ (\ref{isotropies}), such that $V^{\varphi^{-1}}\!\!\!\downarrow^\fg_{\langle t \rangle}$ is not isomorphic to the $\langle t\rangle$-module $nJ_{p^s}$, for any $n, s \ge 0$, where $V = V_{1, \fg}(t)$. Then $\fg$ does not satisfy Property PA. 
    
    In particular, $\Xscr(A, V)  = \{\fp\},$ and $\fp$ is noble for both $\widetilde G, \widetilde G^\varphi$ (the infinitesimal group scheme and its twist), but there is non-isomorphism
    \[V\widetilde \otimes V \not\cong V\widetilde\otimes^\varphi V.\]
\end{lemma}
\begin{proof}
    The support condition $\Xscr(A, V) = \{\fp\}$ is automatic for $V = V_{1, \fg}(t)$ with $t$ a central nilpotent element: For general $t$, it follows from the PBW theorem that $\Xscr(A, V_{1, \fg}(t)) \subseteq \{ \fp (t) \},$ and for central $t$ we have $V_{1, \fg}(t)\downarrow^{\fg}_{\langle t \rangle} = [\fg : \langle t \rangle ]J_1.$ 
    The prime $\fp = \fp(t)$ is noble for $\widetilde G$ by construction, and since we assumed $\varphi \in \Omega(A, \fp)$, we also have, by Lemma \ref{twisting}, that $\varphi^*(\fp) =\fp$ is noble for $\widetilde G^\varphi.$
    
    Suppose $t \in \Nscr(\fg)$ is of order $r$, and let $\fh = \langle t \rangle$ be the $r$-dimensional subalgebra. Since $\fg$ is abelian, say of dimension $n$, we get 
    \[V\downarrow^\fg_\fh = k\uparrow_{\fh}^\fg\downarrow^\fg_\fh = [\fg : \fh] k,\]
    where $[\fg : \fh] =\dim V = p^{n - r}$. By \ref{IndTensor}, we have the polynomial identity \ref{poly1} is satisfied by $V$:
    \[V\widetilde\otimes V \cong [\fg : \fh ] V.\]
    But we assumed that $V^{\varphi^{-1}}\!\!\!\downarrow_{\fh}^\fg$ is not isomorphic to any $nJ_{p^s}$. Supposing that $V^{\varphi^{-1}}$ satisfies the same polynomial identity \ref{poly1}, by restricting we get
    \begin{flalign*}
        [\fg : \fh]V^{\varphi^{-1}}\!\!\!\downarrow_\fh^\fg &\cong (V^{\varphi^{-1}}\widetilde\otimes V^{\varphi^{-1}})\downarrow_\fh^\fg\\
        &\cong (V^{\varphi^{-1}}\!\!\!\downarrow_\fh^\fg) \widetilde \otimes (V^{\varphi^{-1}}\!\!\!\downarrow_\fh^\fg).
    \end{flalign*}
    This contradicts the calculation we gave at the end of \ref{WittVectorExpo}. Now we know $V^{\varphi^{-1}}\widetilde\otimes V^{\varphi^{-1}} \not\cong [\fg : \fh]V^{\varphi^{-1}}$ and therefore, twisting both sides, we have \[V\widetilde\otimes^\varphi V \not\cong [\fg : \fh] V \cong V\widetilde\otimes V.\]
\end{proof}

\begin{cor}\label{automorphismExtension}
Let $\fg'$ be an abelian restricted Lie algebra, and let $B = u(\fg')$ denote the restricted enveloping algebra. Assume there exists $t \in \Nscr (\fg' )$ with associated prime $\fp' = \fp(t) \in \Xscr(B),$ and an isotropy $\psi \in \Omega(B, \fp')$ such that $W^{\psi^{-1}}\!\!\!\downarrow_{\langle t \rangle}^{\fg'}$ is not isomorphic to the $\langle t\rangle$-module $nJ_{p^s}$ for any $n, s \ge 0$,
where $W = V_{1, \fg'}(t).$

Now suppose that $\fg$ is a restricted Lie algebra with $\fg' \subset \fg$ a central Lie subalgebra.  Suppose that there exists an augmented automorphism $\varphi$ of the algebra $u(\fg)$ extending $\psi$.

Let $\fp = \fp(t) \in \Xscr(A)$ be the prime associated to $t \in \Nscr(\fg') \subset \Nscr(\fg)$. 
Then we have that $\varphi \in \Omega(A, \fp)$ is an isotropy such that $V^{\varphi^{-1}}\!\!\!\downarrow^\fg_{\langle t \rangle}$ is not isomorphic to any $nJ_{p^s}$, for any $n, s \ge 0$, where $V = V_{1, \fg}(t)$. 
Further, we have that $\Xscr(A, V) = \{\fp\}$, that $\fp$ is noble for both the infinitesimal group scheme $\widetilde G$ associated to $\fg$ and for its twist $\widetilde G^{\varphi}$, and that 
\[V \widetilde \otimes V \not\cong V\widetilde\otimes^\varphi V,\]
so we may conclude that $\fg$ does satisfy Property PA.
\end{cor}
\begin{proof}
    Let $C$ denote the subalgebra $\psi(u(\langle t \rangle) \subset B$. That $\psi \in \Omega(B, \fp')$ is equivalent to the claim, for any $B$-module $M$, that $M\!\downarrow^B_C$ is projective if and only if $M\!\downarrow^B_{\langle t \rangle}$ is projective. A similar equivalence will work to show $\varphi \in \Omega(A, \fp)$. Since $\varphi$ extends $\psi$, we have $C = \varphi(u(\langle t \rangle))$ also as a subalgebra of $A.$ 
    
    Let $M$ be any $A$-module, and let $N = M\downarrow^A_B$ be its restriction. 
    Assuming $\psi \in \Omega(B, \fp')$, we have immediately that $M\!\downarrow^A_C = N\!\downarrow^B_C$ is projective if and only if $M\!\downarrow^A_{u(\langle t \rangle)} = N\!\downarrow^B_{u(\langle t \rangle)}$ is projective. Thus $\varphi \in \Omega(A, \fp)$. 

    Now we have $V = k\!\uparrow^{\fg}_{\langle t \rangle} = W\!\uparrow^\fg_{\fg'},$ and we have assumed that $W^{\psi^{-1}}\!\!\! \downarrow^{\fg'}_{\langle t \rangle}$ is not isomorphic to any $nJ_{p^s}$. 
    By definition $V^{\varphi^{-1}}$ is the base change of $V$ along $A \xrightarrow{\varphi^{-1}} A$. Since $V = W\!\uparrow^\fg_{\fg'} = A\otimes_{B} W$, we then have $V^{\varphi^{-1}}$ is the base change of $W$ along the composition $B \hookrightarrow A\xrightarrow{\varphi^{-1}} A.$ Since $\varphi$ extends $\psi$, we also have $\varphi^{-1}$ extends $\psi^{-1}$. Hence we have isomorphism 
    $V^{\varphi^{-1}} \cong W^{\psi^{-1}}\!\uparrow_{\fg'}^\fg.$

    By Lemma \ref{generalAutomorphismExtension} we have $V^{\varphi^{-1}}\downarrow^{\fg}_{\langle t \rangle}$ is not isomorphic to any $mJ_{p^s}$, as we know this restriction result is equivalent to the CGM polynomial identity being witnessed by $t$, by Lemma \ref{nPALemma}.

    What remains is to show that $\Xscr(A, V) = \{\fp\}.$ For this, assuming again that $\fg'$ is central in $\fg$, the Nichols-Zoeller basis shows that, supposing $t$ is of rank $r$, $V\downarrow_{\langle t^{[p]^{r-1}}\rangle}$ is not projective. Therefore $\fp = \fp(t)$ belongs to $\Xscr(A, V)$. A simple argument using the PBW basis (reviewed in \ref{PBWsupport}) shows, for any restricted Lie algebra $\fg$, that if $V = k \uparrow^{\fg}_{\langle t \rangle }$ for some nonzero $t \in \Nscr(\fg),$ then $\Xscr(u(\fg), V) \subseteq \{\fp(t)\},$ so we are done.
\end{proof}
\begin{cor}\label{directSumExtension}
Let $\fg'$ be a restricted Lie algebra meeting the same hypotheses as Lemma \ref{nPALemma}, and $\fg''$ any restricted Lie algebra. Then $\fg = \fg' \oplus \fg''$ also meets the hypotheses of Lemma \ref{nPALemma} and therefore $\fg$ does not satisfy Property PA.  
\end{cor}
\begin{proof}
We have isomorphism of algebras $u(\fg) \cong u(\fg')\otimes u(\fg'').$ Therefore an isotropy $\psi \in \Omega(\fg', \fp')$ extends to an automorphism $\varphi : u(\fg) \to u(\fg)$, defined by $\varphi = \psi \otimes u(\fg'').$ Since $\fg'$ is a central subalgebra of $\fg$, we apply Corollary \ref{automorphismExtension}, and get that $fg$ meets the hypothesis of Lemma \ref{nPALemma}. 
\end{proof}

\subsection{Representation type of abelian restricted Lie algebras}\label{repTypesforAbLie}

We begin by recalling the structure theorem for abelian restricted Lie algebras over an algebraically closed field $k$. We denote $\fn_n$ the $p$-nilpotent cyclic Lie algebra of dimension $n$, i.e.
\[\fn_n = \langle x_1, \dots, x_n \mid x_i^{[p]} = x_{i+1}, \text{ where }x_{n+1} = 0\rangle.\]
Denote $\ft = \langle x  \mid x^{[p]} = x\rangle$ the $1$-dimensional torus. 
\begin{thm}\label{SeligmanStructure}(Seligman, 1967 \cite{Sel67}) Let $\fg$ be an ablian restricted Lie algebra of finite dimension over $k$. Then $\fg$ has a direct sum decomposition as 
\[\fg \cong \ft^r \oplus \sum_{i \ge 1} \fn_i^{s_i}\]
for some $r \ge 0$ and finitely many nonzero $s_i \ge 0.$
\end{thm}

For any restricted Lie algebra $\fg$, the algebra $\ft \oplus \fg$ has the same representation type as $\fg$. This is because $u(\ft)$ is isomorphic to a direct product of $p$ many copies of $k$, so $u(\ft \oplus \fg) \cong u(\ft) \otimes u(\fg)$ is a direct product of $p$ many copies of $u(\fg).$ Therefore, in classifying abelian Lie algebras $\fg$ according to representation type, we may reduce to the nullcone $\Nscr(\fg)$, which for $\fg$ abelian is a Lie subalgebra. By Seligman's structure theorem, $\Nscr(\fg)$ is a direct sum of copies of $\fn_n,\ n \ge 1$.  
\begin{thm}
    Let $\fg$ be an abelian restricted Lie algebra of finite dimension over $k$, and let $n$ be the dimension of $\Nscr(\fg)$. 
    \begin{enumerate}[I.]
        \item If $\Nscr( \fg)$ is cyclic (i.e. isomorphic to $\fn_n$), then $\fg$ is of finite representation type. 
        \item If $\Nscr(\fg)$ is not cyclic and $p^n = 4$ (i.e. $p = n = 2$), then $\fg$ is of tame representation type. 
        \item In any other case $p^n > 4$ and we have $\fg$ is of wild representation type. 
    \end{enumerate}
\end{thm}
\begin{proof}
    Since $\Nscr(\fg)$ is a Lie subalgebra with enveloping algebra isomorphic as an associative algebra to the group algebra of some finite abelian $p$-group over $k$, we may appeal directly to modular representation theory of finite groups. It has long been known (see Bondarenko and Drozd \cite{BonDrozd82}) that a group is of wild representation type over $k$ if and only if its Sylow $p$-subgroup is not cyclic, with abelianization of order $> 4,$ with the only tame $p$-groups appearing in characteristic $p = 2$. In particular the only abelian $p$-group of tame representation type is the Klein $4$-group, and any noncyclic abelian $p$-group of order $ > 4$ is of wild representation type. 
\end{proof}

\begin{cor}\label{primordial}
Let $\fg$ be an abelian restricted Lie algebra of wild representation type with no nontrivial wild direct summands (for any decomposition $\fg \cong \fg'\oplus \fg'',$ if $\fg'$ is of wild representation type then $\fg'' = 0$).
\begin{enumerate}[I.]
    \item If $p = 2,$ then $\fg = \fn_1 \oplus \fn_1 \oplus \fn_1,$ or $\fg = \fn_n \oplus \fn_m$ for $n + m \ge 3,$ and $n, m \ge 1$.
    \item If $p > 2,$ then $\fg = \fn_n \oplus \fn_m$ for $n, m \ge 1$.
\end{enumerate}
\end{cor}

\subsection{No wild abelian Lie algebra satisfies Property PC}
We have proven in Theorem \ref{nPA} that no abelian algebra $\fg$ of dimension $2$ with trivial restriction $\fg^{[p]} = 0$ may satisfy Property PA for $p > 2$. We will show how to extend this result to all wild abelian Lie algebras.  
\begin{prop}\label{primordialProp1}
    Let $\fg$ be an abelian restricted Lie algebra of wild representation type, with no nontrivial wild direct summands as in Corollary \ref{primordial}. If $\fg \not\cong \fn_n \oplus \fn_n$ or if $p \neq 2$, then $\fg$ meets the hypotheses of Lemma \ref{nPALemma} for a nilpotent $t \in \Nscr(\fg)$ of order 1: there exists $t \in \Nscr_1(\fg)$ with $u(\langle t \rangle )\hookrightarrow u(\fg)$ representing $\fp \in \Xscr(A)$, and an isotropy $\varphi \in \Omega(u(\fg), \fp)$ such that $V^{\varphi^{-1}}\downarrow^{\fg}_{\langle t\rangle}$ is not trivial (isomorphic to some $nJ_1$) and not projective (isomorphic to some $nJ_{p}$). 
    Therefore $\fg$ does not satisfy Property PA. 
\end{prop}
\begin{proof}
    By Corollary \ref{primordial} we have three cases to consider. But by Lemma \ref{nPA}, we have already covered the case where $p > 2$ and $n = m = 1$ for $\fn_n \oplus \fn_m$. For the remaining cases we may assume that $n + m \ge 3$ in any characteristic. 

    Let $p = 2$ and assume $\fg = \langle x, y, z \mid x^{[2]} = y^{[2]} = z^{[2]} = 0\rangle,$ so $A = u(\fg) = k[x,y,z]/(x^2, y^2, z^2)$. Define $\fh$ to be the subalgebra $\langle x \rangle$ and define $\varphi \in \Aut(A)$ by \[\varphi(x) = x + yz,\quad  \varphi(y) = y,\quad \varphi(z) = z.\] 
    Then $\varphi \in \Omega(A, \fp)$ where $\fp$ is represented by $u(\fh)\hookrightarrow A,$ and $V^{\varphi^{-1}}$ is not annihilated by $x$. Therefore the restriction $V^{\varphi^{-1}}\downarrow^{\fg}_{\fh}$ is neither trivial nor projective. 

    Now assume $\fg = \fn_n \oplus \fn_{m}$, for $n \ge 1$ and $m \ge \max\{n, 2\}.$ Denote $x_{n+1} = y_{m+1} = 0$ and take bases for cyclic summands
    \[\fn_n = \langle x_1, \dots, x_n \mid x_i^{[p]} = x_{i+1}\rangle, \quad \fn_m = \langle y_1, \dots, y_m \mid y_i^{[p]} = y_{i+1}\rangle,\]
    so 
    $A = u(\fg) = k[x, y] / (x^{p^n}, y^{p^m})$
    for $x = x_1, y = y_1.$ We take $\fh$ to be the subalgebra $\langle x_n = x^{p^{n-1}}\rangle.$
    Now define $\varphi \in \Aut(A)$ by 
    \[\varphi(x) = x + y^2,\quad \varphi(y) = y.\]
    Then $\varphi \in \Omega(A, \fp)$, where $\fp$ is represented by $u(\fh) \hookrightarrow A.$
    We have $\varphi(x_n) = (x + y^2)^{p^{n-1}} = x_n + y^{2p^{n-1}}$. But $2p^{n-1} < p^m$ provided either $m > n$ or $p > 2$ so $V^{\varphi^{-1}}$ is not annihilated by $x_n$. Hence $V^{\varphi^{-1}}\!\!\!\downarrow^{\fg}_{\fh}$ is neither trivial nor projective.    
    
\end{proof}
    The remaining case of $p = 2$ and $m = n \ge 2$ has been excluded from the above Proposition, as one finds it is necessary to use a $t \in \Nscr(\fg)$ of order $n$, not simply order 1. Indeed, have $A = u(\fg) = k[x, y]/(x^{2^n}, y^{2^n})$, and any 1-dimensional restricted Lie subalgebra $\fh \subset \fg$ is generated by $t = ax^{2^{n-1}} + by^{2^{n-1}}$ for $a,b$ not both $0$. We may assume $a = 1, b =0$. Any  $\varphi \in A$ fixing the corresponding point $\fp = [1: 0] \in \PP^1,$ must have $\varphi(x) = x + \xi$ for a higher order term $\xi$, in which case $\varphi$ fixes $x^{2^{n-1}} \in A.$ This case is dealt with in the next proposition.


\begin{prop}\label{primordialProp2}
    Let $k$ be an algebraically closed field of characteristic $p = 2$, and let $\fg = \fn_n \oplus \fn_n$ for $n \ge 2$, with $A = u(\fg) = k[x,y] / (x^{2^{n}}, y^{2^{n}}).$ Then $x \in \Nscr_n(\fg)$, so take $\fh = \langle x\rangle$ be the cyclic Lie subalgebra of $\fg$ of dimension $n$. The associated prime $\fp = \fp(x) \in \Xscr(A),$ written in coordinates dual to the basis $x^{2^{n-1}},y^{2^{n-1}}$ for $\Nscr_{1}(\fg),$ is the point $\fp = [1: 0] \in \Xscr(A) = \PP^1.$
    Let $V = k\uparrow_{\fh}^{\fg}$ be the induced module of the trivial $\fh$ module. Then there exists an isotropy $\varphi \in \Omega(A, \fp)$ such that $V^{\varphi^{-1}}$ is not isomorphic to any $nJ_{2^s}$. 
    Therefore $\fg$ does not satisfy Property PA. 
\end{prop}
\begin{proof}
    Define an automorphism $\varphi \in \Aut(A)$ by 
    \[\varphi(x) = x + y^{2^{n-1} - 1}, \quad \varphi(y) = y.\]
    Then $\varphi \in \Omega(A, \fp).$ 
    Now we examine the representation $V^{\varphi^{-1}} \downarrow_{\fh}^\fg,$ which has its decomposition into Jordan blocks determined by the action of $x \in A$ on $V^{\varphi^{-1}}$. 
    The action of $x$ on $V^{\varphi^{-1}}$ is a matrix agreeing with the action of $y^{2^{n-1} - 1}$ on $V,$ the induced module. 
    The module $V = u(\fg) \otimes_{u(\fh)} \otimes k$ has a $k$-linear basis of elements $y^i \otimes 1,$ for $0 \le i \le 2^{n-1}.$
    The Jordan decomposition of $y^{2^{n-1}-1}$ consists of two blocks isomorphic to $J_2,$ with $k$-linear bases $\{y^{2^{n-1}}\otimes 1, y\otimes 1\}$ and $\{y^{2^{n-1}-1} \otimes 1, 1\otimes 1\}$. The other blocks are all isomorphic to $J_1.$ In particular not all the blocks are of the same size so we are done. 
\end{proof}

\subsection{A family of nonabelian wild Lie algebras}
Assume for this section that $p > 2.$ Let $\fg_n$ be the \emph{Heisenberg Lie algebra} of dimension $2n + 1$, having presentation 
\[\fg_n = \langle x_i, y_i, z; 1\le i \le n \mid [x_i, y_j]  = \delta_{ij}z,\ [z,x_i] = [z, y_i] = 0,\ x_i^{[p]} = y_i^{[p]} = z^{[p]} = 0\rangle.\]

We have canonical embeddings $\fg_n \subset \fg_{n+1}$ of Lie algebras by keeping the indexing of $x_i, y_i$ the same. Assuming $p > 2$, we have that each $\fg_n$ is of wild representation type. In this section we will first show that $\fg_1$ does not satisfy Property PA by an argument of polynomial identities {\`a} la Lemma \ref{nPALemma}, and then that this can be extended to any $\fg_n$ via Lemma \ref{generalAutomorphismExtension} (note each subalgebra $\fg_n \subset \fg_{n+1}$ is central).

Since $[[\fg_n, \fg_n], \fg_n] = 0$ and a basis for $\fg_n$ is annihilated by the $[p]$ restriction mapping, it follows that $\fg_n^{[p]} = 0.$ Thus $\Nscr_1(\fg_n) = \fg_n,$ and we may identify $\Xscr(\fg_n)$ with $\PP(\fg_n),$ i.e. the variety $\PP^{2n}$ of $1$-dimensional subspaces of $\fg_n$.

\begin{expo}\label{PBWsupport}
For abelian restricted Lie algebras $\fg$ and their corresponding infinitesimal group schemes $\widetilde G$, the equivalence relation on general $\pi$-points is straightforward. On one hand the structure of cohomology is easier to deal with, using well-known constructions for minimal resolutions. On the other hand even our note \ref{noteOnPiEquivalence} is easier to apply in the case of abelian Lie algebras: the induced module $V = k\uparrow_{\langle t \rangle}^{\fg}$ from a subgroup $\langle t \rangle \in \PP(\Nscr_1(\fg))$ is easily shown to be supported only at the corresponding point $\fp(t) \in \Xscr(\widetilde G)$, by restricting along each subalgebra in $\PP^1(\Nscr_1(\fg))$.
The equivalence class of $\pi$-points corresponding to $\fp ( t )$ is therefore 
$\{ \alpha \mid \alpha^*(V) \text{ is not projective}\}.$ 

The latter approach is adaptable to the following: let $\fg$ be any finite dimensional restricted Lie algebra and $A = u(\fg)$. 
Let $\fp = \fp(t)$ for some nonzero $t \in \Nscr(\fg)$, and let $V = k\uparrow^{\fg}_{\langle t \rangle}$ be the induced module. Say $t$ is of order $r$ and so $\langle t \rangle $ is $r$-dimensional. Identifying $\Xscr(A) = \PP(\Nscr_1(\fg))$ by Friedlander and Parshall \cite{FrPar86}, we want to show that $\Xscr(A, V) \subset \{\fp\}.$
Without loss of generality assume $\{\fp\} \subsetneq \Xscr(A)$, and let $\fq = \fp(s) \in \Xscr(A)$ be distinct from $\fp$, for some nonzero $s \in \Nscr_1(\fg)$. In particular $s, t, t^{[p]}, \dots t^{[p]^{r-1}}$ is a linearly independent set. Extend this to an ordered basis
\[s_1< \dots <  s_n\]
for $\fg$, by assuming $s = s_1$ and $s_{n - i} = t^{[p]^i}$ for $0\le i \le r-1.$ By the PBW theorem, $u(\fg)$ has a $k$-basis of ordered monomials in the coordinates $s_j$, $1 \le j \le n$. Therefore, for the trivial $\langle t \rangle$-module $k$, the induced module \[V = k\uparrow^{\fg}_{\langle t \rangle} = u(\fg) \otimes_{u(\langle t \rangle )} k\]
has a $k$-basis of simple tensors $s^{\alpha}\otimes 1$, for ordered monomials $s^{\alpha}$ in the coordinates $s_j$ for $1 \le j \le n - r.$ Hence $V\downarrow_{\langle s \rangle}$ is free over $u(\langle s \rangle) = k[s]/s^p,$
since
\[s \cdot s_1^{\ell}s^\beta\otimes 1 = s_1^{\ell + 1} s^{\beta}\otimes 1\]
for any ordered monomial $s^{\beta}$ in the coordinates $s_j$ for $2 \le j \le n - r.$ We conclude $\fq \not\in \Xscr(A, V).$

If $V$ is known to be not projective over $u(\fg)$, then we know in fact that $\Xscr(A, V) = \{\fp\},$ as the support must be nonempty. Recall the algebra $\fg$ is called \emph{unipotent} if $u(\fg)$ is a local ring (having a unique maximal left ideal). Whenever $\fg$ is unipotent, the induced module $V$ can not be projective because $V$ has dimension $p^{n - r},$ which is strictly smaller than the rank 1 free module, the smallest projective. If $\fg = \fg' \oplus \mathfrak{u}$ for some unipotent $\mathfrak{u}$, and $t \in \Nscr(\mathfrak{u})$, then the induced module $V$ also can not be projective. Thus if $\fg$ is any abelian restricted Lie algebra, we have another proof that $\Xscr(A, V) =\{\fp\}$ using Seligman's structure theorem, since any sum of $p$-nilpotent cyclic Lie algebras is unipotent. But we have in fact done enough work to compute when two $\pi$-points are equivalent in some important nonabelian cases as well, without resorting to resolutions!
\end{expo}

\begin{expo}\label{indForHeis}(Induced modules for $\fg_1$)
Here we will give matrices describing the induced modules $V_r = V_{1, \fg}(x)$ for $\fg = \fg_1$, where $x = x_1 \in \fg$ and $J_r$ denotes the unique indecomposable $\langle x \rangle$-representation of dimension $r$. We also denote $y = y_1, A = u(\fg)$, and $D = u(\langle x \rangle) \cong k[x]/x^p$. Now we have $J_r \cong D / Dx^r$, where $Dx^r$ is the left-ideal, and hence $V_r \cong A / Ax^r.$ We choose $y < z < x$ as an ordered basis for $\fg$, so ordered monomials $y^iz^jx^\ell,\ \ 0\le i, j, \ell \le p-1,$ are a basis for $A$. The action of $x$ on $A$ is given by 
\[x \cdot y^iz^jx^\ell  = y^iz^jx^{\ell + 1} + iy^{i-1}z^{j+1}x^{\ell}.\]

Now we'll give matrices for the action of $x$ on $V_r,$ to find the $D$-module structure of each $V_r\downarrow^A_{D}.$ We put a lexicographical order on the basis of representing monomials $y^i z^j x^\ell,\quad  0 \le i, j \le p-1,\ \ 0 \le \ell \le r-1,$ 
for $V_r \cong A / Ax^r$ by the following relations:
\begin{flalign*}
    y^{i}z^{j}x^{\ell} < y^{i}z^{j'}x^{\ell} \quad &\iff\quad j > j' \qquad \forall i, \ell,\\
    y^{i}z^{j}x^{\ell} < y^{i'}z^{j'}x^{\ell}\quad &\iff \quad i < i' \qquad \forall j, j',\ell\\
    y^{i}z^{j}x^{\ell} < y^{i'}z^{j'}x^{\ell'}\quad &\iff \quad \ell > \ell' \qquad \forall i, i', j, j'.
\end{flalign*}
Recall our notation $\mathfrak{N}_p$ for the $p\times p$ upper triangular nilpotent Jordan block of rank $p-1$. Now we define the block matrix $\mathfrak{M}$, with $p \times p$ blocks, each of size $p\times p$
\begin{center}
\begin{tikzpicture}
    \node at (-3, 0) {$\fM =$};
    
    \matrix (m) [matrix of math nodes, left delimiter={(}, right delimiter={)}] {
        0 & \mathfrak{N}_p & 0 &\cdots & 0\\
        0 & 0 &2\mathfrak{N}_p & \cdots & 0\\
        \vdots & \vdots & \vdots & \ddots & \vdots\\
        0 & 0 & 0 &0 & (p-1)\mathfrak{N}_p\\
        0 & 0 & 0& 0 & 0\\
    };

    \draw[->] ([yshift=6pt] m-1-1.north west) -- ([yshift=6pt] m-1-5.north east) 
        node[midway,above] {$i$};
\end{tikzpicture}
\end{center}

Then with respect to our ordered basis of monomials, the matrix representing the action of $x$ on $V_r$ is given by the upper triangular block matrix $\mathfrak{L}_r$, with $r\times r$ blocks, each of size $p^2 \times p^2$

\begin{center}
\begin{tikzpicture}
    \node at (-2.5, 0) {$\fL_r =$};
    
    \matrix (m) [matrix of math nodes, left delimiter={(}, right delimiter={)}] {
    \mathfrak{M} & I_{p^2} & 0 &\cdots & 0\\
    0 & \mathfrak{M} & I_{p^2} & \cdots &0\\
    \vdots & \vdots & \vdots & \ddots & \vdots\\
    0 & 0 & 0 & \mathfrak{M} & I_{p^2}\\
    0 & 0 & 0 & 0 & \mathfrak{M}\\
    };

    \draw[->] ([xshift = 10pt,yshift=6pt] m-1-5.north west) -- ([xshift = -10pt, yshift=6pt] m-1-1.north east) 
        node[midway,above] {$\ell$};
\end{tikzpicture}
\end{center}
where $I_{p^2}$ is the identity matrix. Column positions of $\fL_r$ correspond to a set of monomials $y^iz^jx^\ell$ for fixed $\ell,$ increasing to the left. Within $\ell$th column, each column (of e.g. $\fM$ or $I_{p^2}$) is a fixed $i$, increasing to the right. 

In Proposition \ref{HeisenbergnPA}, we will also need the matrix $\mathfrak{O}_r$ representing the action of $(yz)^{p-1}\in A$ on $V_r.$
First define $\mathfrak{E}$ to be the $p\times p$ matrix
\begin{center}
\begin{tikzpicture}
    \node at (-1.7, 0) {$\fE =$};
    
    \matrix (m) [matrix of math nodes, left delimiter={(}, right delimiter={)}] {
    0 & \cdots & 0 & 1\\
    \vdots & \ddots & \vdots & \vdots\\
    0 & \cdots & 0 & 0\\
    0 & \cdots & 0 & 0\\
    };

    \draw[->] ([xshift = 10pt,yshift=6pt] m-1-4.north west) -- ([xshift = -10pt, yshift=6pt] m-1-1.north east) 
        node[midway,above] {$j$};
\end{tikzpicture},
\end{center}
and define the block matrix $\mathfrak{F}$, with $p\times p $ blocks, each of size $p \times p$, by
\begin{center}
\begin{tikzpicture}
    \node at (-1.7, 0) {$\fF =$};
    
    \matrix (m) [matrix of math nodes, left delimiter={(}, right delimiter={)}] {
            0 & 0 & \cdots & 0\\
            0 & 0  & \cdots & 0\\
            \vdots & \vdots & \ddots & \vdots\\
            \mathfrak{E} & 0 & \cdots & 0\\
    };
    \draw[->] ([yshift=6pt] m-1-1.north west) -- ([yshift=6pt] m-1-4.north east) 
        node[midway,above] {$i$};
\end{tikzpicture}.
\end{center}
Now we have $\mathfrak{O}_r$ is a block matrix, with $r \times r$ blocks, each of size $p^2 \times p^2,$ given by
\begin{center}
\begin{tikzpicture}
    \node at (-2, 0) {$\fO_r =$};
    
    \matrix (m) [matrix of math nodes, left delimiter={(}, right delimiter={)}] {
            \mathfrak{F} & 0 & \cdots & 0\\
            0 & \mathfrak{F}  & \cdots & 0\\
            \vdots & \vdots & \ddots & \vdots\\
            0 & 0 & \cdots & \mathfrak{F}\\
    };

    \draw[->] ([xshift = 10pt,yshift=6pt] m-1-4.north west) -- ([xshift = -10pt, yshift=6pt] m-1-1.north east) 
        node[midway,above] {$\ell$};
\end{tikzpicture}.
\end{center}
\end{expo}

\begin{prop}\label{HeisenbergnPA}
    Let $\fg = \fg_1$ be the Heisenberg Lie algebra of dimension 3. Let $A =u(\fg)$. There exists $t \in \Nscr(\fg)$ (of order 1), such that, for $\fp = \fp(t) \in \Xscr(A),$
    there exists an isotropy $\varphi \in \Omega(A, \fp)$, for which $V \widetilde \otimes V \not\cong V\widetilde\otimes^{\varphi} V$, where $V = V_{1, \fg}(t)$. 
    Further, $\fg$ is unipotent and therefore $\Xscr(A, V) = \{\fp\}$. So $\fg$ does not satisfy Property PA.
\end{prop}
\begin{proof}
    Denote $x = x_1, y = y_1$ and so we have 
    \[A = \frac{k\langle x, y, z\rangle}{([x,y] - z, [z,x], [z, y], x^p, y^p, z^p)}.\]
    Define an automorphism $\varphi \in \Aut(A)$ by 
    \[\varphi(x) = x + (yz)^{p-1}; \quad \varphi(y) = y.\]
    Note that $\varphi(z) = z$, and the inverse is given by 
    \[\varphi^{-1}(x) = x - (yz)^{p-1}; \quad \varphi^{-1}(y) = y.\]
    We will take $t = x \in \Nscr_1(\fg)$, and argue that $\varphi \in \Omega(A, \fp)$, for $\fp = \fp(x)$ by following the arguments \ref{noteOnPiEquivalence} and \ref{PBWsupport}. So let $V = k\uparrow_{\langle x \rangle}^\fg$ be the induced module. Then $\Xscr(A, V) = \{\fp\}$ (since $\fg$ is unipotent) and so $\varphi \in \Omega(A, \fp)$ if and only if the restriction $V\!\downarrow_B^A$, to the image 
    subalgebra $B = \varphi(u(\langle x \rangle ))$, is not a projective module. 

    By the PBW theorem using the ordered basis $y < z < x$ for $\fg$, we get the following basis for $V = A \otimes_{u(\langle x \rangle) } k$ 
    \[\{y^iz^j\otimes 1 \mid 1 \le i,j, \le p-1\}.\]
    The subspace $\langle 1 \otimes 1, (yz)^{p-1} \otimes 1\rangle$ is a Jordan block of size 2 for the element $(yz)^{p-1} \in A$ acting on $V$. One checks that $x \in A$ annihilates both $1 \otimes 1$ and $(yz)^{p-1}\otimes 1$, and neither are in the image of $x$. It follows that $V\downarrow^{A}_B$ has $J_2$ as a direct summand, as the action of $\varphi(x) = x + (yz)^{p-1}$ contains a Jordan block of size 2. We conclude $V\downarrow^{A}_B$ is not projective and hence $\varphi \in \Omega(A, \fp).$

    Now, we want to argue that $V \widetilde\otimes^{\varphi} V \not\cong V \widetilde \otimes V.$ Frobenius reciprocity (\ref{IndTensor}) gives us
    \[V \widetilde \otimes V \cong V\downarrow_{\langle x \rangle}^{\fg} \uparrow_{\langle x \rangle}^{\fg}.\]
    Recall our notation $J_i$ to mean the unique indecomposable $\langle x \rangle$-module of dimension $i$ for $1 \le i \le p$. We can compute directly with the basis $\{y^iz^i \otimes 1\}$ for $V$ that we have a Mackey decomposition \begin{equation}\label{HeisPolyId}V\downarrow^{\fg}_{\langle x \rangle} \cong 2J_1 \oplus 2J_2 \oplus \dots \oplus 2J_{p-1} \oplus J_p,
    \end{equation}and therefore a polynomial identity
    \[V_1\widetilde\otimes V_1 \cong 2V_1 \oplus 2V_2 \oplus \dots \oplus 2V_{p-1} \oplus V_p,\]
    the $1^{st}$ CGM polynomial identity for $\fg$ at $x$,
    where $V_i = V_{i, \fg}(x)$ (note $V = V_1$ and $V_p$ is the free $A$-module of rank $1$). All we must do now is verify that the same polynomial idenity fails after twisting by $\varphi^{-1}$, i.e. that 
    \begin{equation}\label{HeisNonIso}
    V_1^{\varphi^{-1}}\widetilde \otimes V_1^{\varphi^{-1}} \not\cong 2V_1^{\varphi^{-1}} \oplus 2V_2^{\varphi^{-1}} \oplus \dots \oplus 2V_{p-1}^{\varphi^{-1}} \oplus V_p^{\varphi^{-1}}.\end{equation}
    This will follow from seeing that the number of $J_1$ blocks in $(V_1^{\varphi^{-1}} \widetilde\otimes V_1^{\varphi^{-1}})\downarrow_{\langle x \rangle }^{\fg}$ differs from that of the restriction of the right hand side. 
    By restriction property for $\widetilde\otimes$, we have
\begin{flalign*}
    (V_1^{\varphi^{-1}} \widetilde\otimes V_1^{\varphi^{-1}})\downarrow_{\langle x \rangle }^{\fg} &\cong V_1^{\varphi^{-1}}\downarrow_{\langle x \rangle }^{\fg} \widetilde \otimes V_1^{\varphi^{-1}}\downarrow_{\langle x \rangle }^{\fg}\\
    &= V_1 \downarrow^A_B \widetilde \otimes V_1\downarrow^{A}_B.
\end{flalign*}
    Here we use $\widetilde \otimes$ to also denote the tensor product of $\langle x \rangle$-representations. We have also abused notation to assert that $V_1^{\varphi^{-1}}\downarrow_{\langle x \rangle }^{\fg} = V_1 \downarrow^A_B$, identifying $B$ modules as $\langle x \rangle$-representations using the fixed isomorphism $\varphi : u(\langle x \rangle ) \to B.$ 
    
    Let $L_1 = V_1\downarrow^{A}_{B}$. We compute the number of $J_1$ summands in $L_1\widetilde\otimes L_1$ directly as follows: The matrix representing the action of $x$ on $L_1$ is given by $\mathfrak{L}_1 + \mathfrak{O}_1$, as defined in \ref{indForHeis}. The Jordan canonical form of $\mathfrak{L}_1 + \mathfrak{O}_1$ will then tell us the number of summands $J_r$ for each $1 \le r \le p$, within $L_1,$ say
    \[L_1 \cong \sum_{r=1}^{p}c_r J_r.\]
    Then the number of $J_1$-blocks in $L_1\widetilde\otimes L_1$ is $\sum_{r = 1}^{p-1}c_r^2,$ a standard computation for $\langle x \rangle$-representations. 

    Note now, in $\ref{indForHeis}$, for $r = 1$, we have $\fL_1 = \fM$ and $\fO_1 = \fF$, so we have 
    \[\fL_1 + \fO_1 = \fM + \fF = \begin{pmatrix}
                0 & \mathfrak{N}_p & 0 &\cdots & 0\\
        0 & 0 &2\mathfrak{N}_p & \cdots & 0\\
        \vdots & \vdots & \vdots & \ddots & \vdots\\
        0 & 0 & 0 &0 & (p-1)\mathfrak{N}_p\\
        \mathfrak{E} & 0 & 0& 0 & 0
    \end{pmatrix}.\]
    Notice $(\fM + \fF)^n = \fM^n$ for each power $n \ge 2$. The rank $\rho_n$ of $(\fM + \fF)^n$ is therefore
    \[\rho_1 = (p-1)^2 + 1;\quad \rho_n = (p - n)^2,\ \ n \ge 2.\]
    We conclude $L_1 \cong 3J_2\oplus 2J_3 \oplus \dots \oplus 2J_{p-1} \oplus J_{p},$ and so $L_1\widetilde \otimes L_1$ has $9 + 4(p-3)$ many $J_1$ summands. 

    On the other side of the polynomial identity, we want to find the number of $J_1$ summands in each $L_r = V_r\downarrow_{B}^A,$ for the remaining values $r >1$. The action of $x$ on $L_r$ is given by the matrix $\fL_r + \fO_r$. So the number of $J_1$ summands is the nullity of $\fL_r + \fO_r$, minus the number of its Jordan blocks of size $>1.$ 
    
    We find $\fL_r + \fO_r$ is given by the $r\times r$ block matrix, with blocks of size $p^2\times p^2$
    \[\fL_r + \fO_r = \begin{pmatrix} 
    \fM + \fF & I_{p^2} & 0 & \cdots & 0\\
    0 &\fM + \fF & I_{p^2} & \cdots & 0\\
    \vdots & \vdots & \vdots &\ddots & \vdots\\
    0 & 0 & 0& \fM + \fF & I_{p^2}\\
    0 & 0 & 0 & 0 & \fM + \fF
    \end{pmatrix}. 
    \]
    Noting again $(\fM + \fF)^2 = \fM^2$, we have $(\fL_r + \fO_{r})^2$ is the block matrix
    \[(\fL_r + \fO_{r})^2 = \begin{pmatrix} 
\fM^2 & \fM + \fF & I_{p^2} & 0 & \cdots & 0\\
0       & \fM^2 &  \fM + \fF & I_{p^2} & \cdots & 0\\
0       & 0     & \fM^2     & \fM + \fF & \ddots & 0\\
\vdots  & \vdots & \vdots   & \vdots    &\ddots & \cdots\\
0       & 0     & 0         & \fM^2     & \fM + \fF & I_{p^2}\\
0 & 0& 0& 0& \fM^2 & \fM + \fF\\
0 & 0& 0&0&0&\fM^2
\end{pmatrix}.
\]
For the rank of $\fL_r + \fO_r$: We may construct a similar matrix, which is a diagonal $p\times p$ matrix, with blocks of size $rp\times rp,$ by using a partition of the blocks of size $p\times p$. Each $pr\times pr$ diagonal entry we will call a string. Each string has $2r -1$ nonzero block entries of size $p\times p$. There are $r$ strings containing $\fE$, in the form
\[
\begin{pmatrix}
    \fE & \hspace{-0.5cm} I_p \\
    & \hspace{-0.5cm} (p-1)\fN_p & \hspace{-0.5cm} I_p \\
    && \hspace{-0.5cm} (p-2)\fN_p & \hspace{-0.5cm} I_p \\
    && \hspace{-0.5cm} \ddots & \hspace{-0.5cm} (p -(r -1))\fN
\end{pmatrix},
\begin{pmatrix}
    \fN & I_p \\
    & \fE & \hspace{-0.5cm} I_p \\
    && \hspace{-0.5cm} (p-1)\fN_p & \hspace{-0.5cm} I_p \\
    && \hspace{-0.5cm} \ddots & \hspace{-0.5cm} (p - (r-2))\fN_p
\end{pmatrix},
\]
\[
\begin{pmatrix}
    2\mathfrak{N_p} & I_p \\
    &  \mathfrak{N_p} & I_p \\
    && \mathfrak{E} & \hspace{-0.2cm} I_p \\
    &&\ddots& \hspace{-0.2cm}(p-(r-3))\mathfrak{N_p}
\end{pmatrix}, \, \dots\, ,
\begin{pmatrix}
    (r-1)\mathfrak{N}_p & I_p\\
    &\ddots & I_p\\
    &&\mathfrak{N}_p & I_p\\
    &&&\fE
\end{pmatrix}.
\]
The remaining $p-r$ strings, not containing $\fE$, are of the form
\[\begin{pmatrix}
    i\fN_p& \hspace{-0.5cm} I_p\\
    & \hspace{-0.5cm}(i -1)\fN_p & \hspace{-0.5cm}I_p\\
    &&\hspace{-0.5cm}(i - 2)\fN_p & I_p\\
    &&\ddots & \hspace{-0.5cm} (i - r + 1)\fN_p
\end{pmatrix}\]
for $i = r, \dots, p-1$. 
The rank of $\fL_r + \fO_r$ is the sum of the ranks of the $p$ strings, each a matrix of size $pr \times pr.$ 
Each of the strings not containing $\fE$ has rank $r(p-1).$ Each string containing $\fE$ has rank $(r-1)p.$ Thus the rank of $\fL_r + \fO_r$ is $(p-r)r(p-1) + r(r-1)p = rp^2 - 2rp + r^2,$ and the nullity of $\fL_r + \fO_r$ is $2rp - r^2.$


For $r = 2$ the rank of $(\fL_2 + \fO_2)^2$ is $2(p-3)(p-2) + 2(p-1) + 1.$

For $r \ge 3,$ we find the rank of $(\fL_r + \fO_r)^2$ again using strings, which are still $r \times r$ block matrices, with blocks of size $p\times p$. We abbreviate $c(n) = n(n-1)$. There are $(r-1)$ strings containing $\fE$, each of the form
\[\begin{pmatrix}
    0 &\fE & I_p \\
     & 0& (p-1)\fN_p & I_p\\
     & &c(p-1)\fN^2_p &(p-2)\fN_p & \hspace{-0.8cm}\ddots\\
     & & & c(p-2)\fN^2_p & \hspace{-0.8cm}\ddots\\
     & & &\ddots & \hspace{-0.8cm}c(p-(r-1))\fN_p^2 & (p - (r-2))\fN_p\\
    &&&&&c(p-(r-2))\fN^2_p
\end{pmatrix},\]
\[\begin{pmatrix}
     2\fN_p^2& \fN_p& I_p \\
     & 0& \fE & I_p\\
     & &0 &(p-1)\fN_p & \hspace{-0.8cm}\ddots\\
     & & & c(p-1)\fN^2_p & \hspace{-0.8cm}\ddots\\
     & & &\ddots & \hspace{-0.8cm}c(p-(r-2))\fN_p^2 & (p - (r-3))\fN_p\\
    &&&&&c(p-(r-3))\fN^2_p
\end{pmatrix},\]
\[\dots, \begin{pmatrix}
    c(r-1)\fN^2_p & (r-2)\fN_p & I_p &\ddots\\
        & c(r-2)\fN^2 & (r-3)\fN_p & \ddots\\
        &\ddots& 2\fN_p^2 & \fN_p & I_0\\
        && &0 &\fE\\
        && & &0
\end{pmatrix}.\]
Each of these matrices has a rank of $(r-2)p.$

The remaining $p - (r-1)$ strings don't contain $\fE$.
Provided $r\le p-1$, there are $p-(r+1)$ such strings which have no zero blocks along the diagonal. 
These strings have a rank of $r(p-2).$ 
There are $2$ strings having a zero block either in the top left or bottom right entry. 
Still assuming $r \ge 3$, these types have a rank of $(r-1)(p - 1) - 1$ when $r$ is divisible by $3$, and $(r-1)(p-1)$ when $r$ is not divisible by $3$. 

If $p = r,$ the module $L_r$ is projective so we already know the Jordan decomposition. The rank of $(\fL_p +\fO_p)^2$ is hence $p^2(p-2).$

We have now the rank of $(\fL_r + \fO_r)^2$, for $r \le p$ is given by
\[\rho(r, p) = \begin{cases}
    2p^2 - 8p + 11 & r = 2,\\
    rp^2 - 2pr + r^2 - 4r & 3 \le r < p, \quad 3 \divides r,\\
    rp^2 - 2pr + r^2 - 4r + 2& 3 \le r  < p,\quad 3 \not\divides r,\\
    p^2(p-2)  & r=p.
\end{cases}\]
The nullity (i.e. the total number of Jordan blocks) of $(\fL_r + \fO_r)$ was calculated as $2rp - r^2$ for any $r,$ and $ (rp^2 - 2rp + r^2) - \rho(r, p)$ is giving the number of Jordan blocks of size $> 1.$ Hence $L_r$ contains $\tau(r, p)$ many $J_1$ summands, where 
\[\tau(r, p)= 4pr - 2r^2 - rp^2 + \rho(r, p) = \begin{cases}
    3                   & r = 2\\
    (2p-4)r - r^2      & 3 \le r < p,\quad 3 \divides r,\\
    (2p-4)r -  r^2 + 2  & 3 \le r < p, \quad 3 \not\divides r,\\
    0               & r = p = 3.
\end{cases}\]
The number of $J_1$ summands in the right hand side of the $1^{st}$ CGM polynomial for $\rho_i = L_i$ \ref{HeisPolyId}\[2L_1 \oplus 2L_2 \oplus \cdots \oplus 2 L_{p-1} \oplus L_p\] is therefore the sum
\[2\sum_{r = 2}^{p-1}\tau(r, p) =\begin{cases}
    6 & p =3\\
    36m^3 - 3m^2 - 35m + 20 & p = 3m+ 1\\
    36m^3 + 27m^2 - 29m + 10& p = 3m + 2.
\end{cases} \]
and is never equal to $9 + 4(p-3).$

We conclude the non-isomorphism \ref{HeisNonIso}, as the $1^{st}$ CGM polynomial identity is not witnessed by $x$ for $\rho_i = V_i^{\varphi^{-1}}$. Therefore
\[V\widetilde \otimes V \not\cong V\widetilde\otimes^{\varphi} V,\]
and so $\fg$ does not satisfy Property PA. 
\end{proof}

\begin{cor}
     For any $n \ge 1,$ the Heisenberg Lie algebra $\fg_n$ of dimension $2n + 1$ does not satisfy Property PA. 
\end{cor}
\begin{proof}
    We have calculated the $1^{st}$ Mackey coefficients at $x = x_1 \in \fg_1$ as 
    \[c_{i, \fg_1}^1(x) = \begin{cases}
        2 & 1 \le i < p\\
        1 & i = p.
    \end{cases}\]
    For each $n$, define an augmented automorphism $\varphi_n$ of $A_n = u(\fg_n)$ by
    \begin{flalign*}
        \varphi_n : x_1 &\mapsto x_1 + (y_1z_1)^{p-1},\\
        x_i & \mapsto x_i \quad 2 \le i \le n,\\
        y_i &\mapsto y_i \quad 1 \le i \le n,
    \end{flalign*}
    hence $\varphi(z) = z.$
    Each $\fg_n$ contains $\fg_1$ as a central subalgebra, and $\varphi_{n}$ extends $\varphi_1.$ We see $\varphi_{n} \in \Omega(A_n, \fp(x_1) )$ for each $n$: since each $\fg_n$ is unipotent, we know $\Xscr(A_n, V_{1, \fg}(x_1)) = \{\fp(x_1)\},$ and we already showed that $V_{1, \fg_1}(x_1)^{\varphi_1^{-1}}\downarrow^{\fg_1}_{\langle x_1\rangle}$ is not projective, hence $\varphi_1^*(\fp(x_1)) = \fp(x_1).$ For general $n$ we have $V_{1, \fg_n}(x_1) = V_{1, \fg_1}(x_1)\uparrow_{\fg_1}^{\fg_n},$ so copying our base-change argument from Corollary \ref{automorphismExtension} we have $V_{1, \fg_n}(x_1)^{\varphi_n^{-1}} \cong V_{1, \fg_1}(x_1)^{\varphi_1^{-1}}\uparrow_{\fg_1}^{\fg_n}.$ Hence
    \begin{flalign*}
        V_{1, \fg_n}(x_1)^{\varphi_n^{-1}}\downarrow^{\fg_n}_{\langle x_1 \rangle} &\cong V_{1, \fg_n}(x_1)^{\varphi_n^{-1}}\downarrow^{\fg_n}_{\fg_1}\downarrow^{\fg_1}_{\langle x_1 \rangle}\\
        &\cong V_{1, \fg_1}(x_1)^{\varphi_1^{-1}}\uparrow_{\fg_1}^{\fg_n}\downarrow^{\fg_n}_{\fg_1}\downarrow^{\fg_1}_{\langle x_1 \rangle}\\
        &\cong[\fg_n : \fg_1]V_{1, \fg_1}(x_1)^{\varphi_1^{-1}}\downarrow^{\fg_1}_{\langle x_1\rangle},
    \end{flalign*}
    so $\varphi^{*}_{n}(\fp(x_1)) = \fp(x_1).$

    By Proposition \ref{HeisenbergnPA}, we have that the $1^{st}$ CGM polynomial identity is not witnessed by $x_1$ for $\rho_i = V_{i, \fg_1}(x_1)^{\varphi_1^{-1}}.$ By Lemma \ref{generalAutomorphismExtension}, we then have that the $1^{st}$ CGM polynomial identity is not witnessed by $x_1$ for $\rho_i = V_{i, \fg_n}(x_1)^{\varphi_n^{-1}}.$

    Therefore, $V = V_{1, \fg_n}(x_1)$ is a representation of $\fg_n$, with $\Xscr(A_n, V) = \{\fp(x_1)\},$ such that
    \[V\widetilde\otimes V \not\cong V \widetilde\otimes^{\varphi_n} V.\]
    Since $\fp(x_1)$ is noble for $\widetilde G_n$ and $\widetilde G_n^{\varphi_n}$ (the infinitesimal group scheme for $\fg_n$ and its twist), we have that $\fg_n$ does not satisfy Property PA. 
    
\end{proof}

\printbibliography
\end{document}